\documentclass[journal,twoside,web]{ieeecolor}
\usepackage{common/sty/tac}
\usepackage{generic}
\begin{document}
\title{
  Contraction-Based Model Predictive Control\\ using Bilinear Koopman Realizations\\ with Proportional Error Bounds
}
\author{Yuki Higuchi and Kazuhiro Sato
  \thanks{This work was supported by JSPS KAKENHI Grant Number 26K03232.}%
  \thanks{Y. Higuchi and K. Sato are with the Department of Mathematical Informatics, Graduate School of Information Science and Technology, The University of Tokyo, Tokyo 113-8656, Japan, email: higuchi-yuki@g.ecc.u-tokyo.ac.jp (Y. Higuchi), kazuhiro@mist.i.u-tokyo.ac.jp (K. Sato)}
}
\markboth{\hskip25pc IEEE TRANSACTIONS AND JOURNALS TEMPLATE}{Author \MakeLowercase{\textit{et al.}}: Title}

\maketitle
\begin{abstract}
  Data-driven model predictive control based on Koopman operator theory is a promising approach for constrained control of nonlinear systems with unknown dynamics.
  This paper proposes a robust model predictive control framework for such systems using bilinear Koopman realizations with state- and input-dependent proportional approximation-error bounds that vanish at the target equilibrium.
  Since prediction in lifted coordinates with a finite-dimensional Koopman realization need not remain on the manifold of valid lifted states, multi-step prediction may leave the region where one-step error certificates apply.
  We avoid this difficulty by reprojecting each predicted lifted state onto the original state coordinates and lifting it again, yielding an error-aware discrete-time control-affine predictor in the original state space without assuming invariance of the Koopman dictionary.
  For this predictor, we develop a homothetic tube construction based on a discrete-time robust control contraction metric whose radius explicitly accounts for the proportional approximation error bounds.
  The tube tightening handles arbitrary continuously differentiable nonlinear constraints, and the resulting model predictive control problem includes terminal ingredients and a tube-radius penalty that exploits the vanishing uncertainty near the target.
  We prove robust satisfaction of the original nonlinear constraints, recursive feasibility, and exponential stability of the sampled true closed-loop system with high probability over the training data without requiring a globally optimal solution to the model predictive control problem.
  Numerical examples, including nonlinear obstacle-avoidance constraints, demonstrate robust stabilization and higher performance of the proposed approach compared to existing Koopman-based model predictive control methods in terms of smaller closed-loop cost and flexibility.
\end{abstract}

\section{Introduction}\label{section:introduction}

Model predictive control (MPC)~\cite{MRRS00} is widely used for constrained dynamical systems because it can optimize performance while enforcing state and input limits.
Its effectiveness, however, depends on the availability of a predictor that is accurate enough for closed-loop decision making.
For systems whose dynamics are difficult to derive from first principles, such as soft robots~\cite{ABR15}, this requirement naturally leads to predictors learned from data.

\subsubsection*{Data-driven Bilinear Koopman Realizations}
Koopman operator theory~\cite{Koopman31} provides a systematic way to build such predictors: the state is lifted through a nonlinear function, which is called a lifting map,
and a finite-dimensional realization is fitted for the lifted variables.
Finite-dimensional linear realizations are attractive because they convert nonlinear prediction into a linear lifted-space realization, but their approximation quality can be limited~\cite{KY22a}.
Bilinear Koopman realizations offer a richer alternative~\cite{Williams+16, Surana16}: they can approximate control-affine dynamics arbitrarily well under suitable richness assumptions, and they often outperform linear realizations in finite-data experiments for general nonlinear systems~\cite{BFV20}.
Moreover, recent identification methods provide state- and input-dependent proportional error bounds for data-driven bilinear realizations~\cite{Strasser+25, Strasser+26}.
Since these bounds scale with the lifted-state and input magnitudes and vanish at the target equilibrium, they provide structural information that is directly useful for robust controller design.
Because finite data and a finite-dimensional approximation inevitably leave residual errors, the learned Koopman realization should be treated as uncertain.
This leads to robust MPC (RMPC), where the optimization is designed to account explicitly for these errors.

\subsubsection*{Bilinear Koopman Realizations and Multi-Step Prediction}
The main technical obstruction for RMPC using bilinear Koopman realizations is the consistency of multi-step prediction.
The range of the lifting map is generally a nonlinear manifold in the lifted coordinates, and a finite-dimensional learned realization need not map this manifold into itself.
If MPC propagates in the lifted space directly, later prediction may be evaluated at off-manifold points.
Therefore, in the context of combining RMPC, it is necessary to evaluate the error when making predictions from off-manifold points.
For linear realizations, this issue can be circumvented by exploiting their linear structure~\cite{Zhang+22,MCV22,KTZS25}; however, it remains challenging for bilinear realizations.
Koopman dictionary invariance~\cite{GP22} removes this issue~\cite{Worthmann+24,BGSW25} only under a restrictive structural assumption, whereas penalizing deviations from the manifold~\cite{KY22a, KY22b, KY23} does not by itself yield hard guarantees.
Another approach~\cite{Xiong+25} has also shown empirical effectiveness but lacks theoretical guarantees.
We instead restore consistency by projecting the lifted prediction onto the original-state coordinates and re-lifting the result at each prediction step.
This construction yields an error-aware predictor in the original state space with a discrete-time control-affine uncertainty description.

\subsubsection*{Contraction-Based Tube MPC}
For uncertain control-affine systems, control contraction metrics (CCM)~\cite{MS15} and robust CCM (RCCM)~\cite{Zhao+22} give differential conditions under which closed-loop trajectories contract in a state-dependent metric.
RCCM-based MPC schemes~\cite{SZK23, GSS24, ZS24} use this contraction structure to build tubes around nominal trajectories, guaranteeing robust constraint satisfaction without solving a min--max problem online.
This makes tube MPC (TMPC) a suitable basis for the proposed Koopman-based predictor, whose uncertainty size depends on the current state and input.

\subsubsection*{Related Work}
The idea of projecting the lifted prediction onto the original-state coordinates and re-lifting the result at each prediction step has been proposed in several recent studies.
However, one of these methods does not incorporate state constraints into the MPC optimization problem~\cite{Schimperna+25}, whereas another uses a constant error bound, which can render the tightened constraint set empty as the prediction horizon increases~\cite{Schimperna+26}. The former guarantees exponential stability (ES), whereas the latter establishes asymptotic stability (AS) of the sampled closed-loop systems; both guarantees require the MPC optimization problem to be solved exactly.
A recent work~\cite{JL26} also proposes an RCCM-based TMPC for bilinear Koopman realizations with approximation errors and propagates the predictor directly in the lifted space. As mentioned above, direct propagation in the lifted space may evaluate the predictor at points that do not lie on the manifold. To account for off-manifold operation, the work augments valid lifted samples with small perturbations when learning RCCM\@. That approach does not explicitly exploit proportional bounds and instead relies on a constant bound, yielding input-to-state stability (ISS) guarantees rather than ES\@.
Furthermore, constructing a terminal set, that guarantees recursive feasibility, requires an additional assumption other than the RCCM condition and
is challenging to construct in practice, as acknowledged by the authors~\cite[Sec. VII]{JL26}.

\subsubsection*{Contributions}
This paper develops an RMPC framework that combines data-driven bilinear Koopman realizations with proportional error bounds and contraction-based TMPC\@.
The contributions are as follows.
\begin{itemize}
  \item We use a bilinear Koopman realization with a proportional approximation-error bound~\cite{Strasser+26} and the re-lifting described above to obtain an error-aware predictor in the original coordinates.
  \item We construct a discrete-time RCCM-based homothetic tube for this predictor and formulate the corresponding TMPC problem, building on~\cite{SZK23,GSS24,ZS24}.
  \item We prove robust satisfaction of the original nonlinear constraints, recursive feasibility, and exponential stability of the sampled true closed-loop system without requiring a globally optimal solution to the MPC problem.
\end{itemize}
To the best of our knowledge, this is the first RMPC method to explicitly exploit proportional error bounds for bilinear Koopman realizations.
Once a discrete-time RCCM is available, the terminal ingredients admit an explicit construction.
Theoretical comparison with related work is summarized in~\cref{table:comparison}.
\begin{table*}[t]
  \centering
  \caption{Comparison with related work}
  \label{table:comparison}
  \begin{tabular}{c|cccc} 
    \hline 
    Method                      & Prediction Domain & Error Bound used in MPC & Robust Mechanism      & Main Guarantee               \\ \hline 
    \mbox{\cite{Schimperna+25}} & State Space       & -                       & -                     & ES (requires global optimum) \\
    \mbox{\cite{Schimperna+26}} & State Space       & Constant                & Direct Set Tightening & AS (requires global optimum) \\
    \mbox{\cite{JL26}}          & Lifted Space      & Constant                & RCCM Tube             & ISS                          \\
    Proposed                    & State Space       & Proportional            & RCCM Tube             & ES                           \\
    \hline 
  \end{tabular}
\end{table*}
Numerical examples for an inverted pendulum demonstrate robust stabilization and higher performance of the proposed approach compared to existing Koopman-based MPC methods in terms of smaller closed-loop cost and flexibility.
Furthermore, we demonstrate the effectiveness of the proposed method for strict nonlinear obstacle-avoidance constraints in an omni-directional robot example.
Although existing methods~\cite{Worthmann+24,BGSW25,Xiong+25, Schimperna+26} can theoretically handle nonlinear constraints, their numerical examples consider only linear constraints.
Their construction of error-aware constraint tightening relies on computing Minkowski sums of sets. Computing the required set operations can be challenging for general nonlinear constraints.
In contrast, a distinguishing feature of the proposed method is that it enables an explicit construction of error-aware constraint tightening even for nonlinear constraints by the design of an RCCM-based tube.

\subsubsection*{Comparison with the Conference Version}
Compared with the conference version~\cite{HS26}, this paper uses proportional approximation-error bounds and the associated tube design to upgrade the guarantees from practical convergence to ES and to strengthen the flexibility for nonlinear constraints.

\subsubsection*{Organization}
The remainder of this paper is organized as follows.
\Cref{section:preliminary} reviews the problem setup, SafEDMD~\cite{Strasser+26}, and TMPC.
\Cref{section:proposed_method} derives the error-aware predictor and presents the discrete-time RCCM-based TMPC algorithm with theoretical analysis.
\Cref{section:numerical_experiment} provides numerical validation.

\subsubsection*{Notation}
For integers $a$ and $b$, we denote
$\integerset{a,b}=\{a,a+1,\ldots,b\}$.
Let $\mathbb S_+^n$ denote the set of $n\times n$ symmetric positive definite
matrices. We write $A\preceq B$ when $B-A$ is positive semi-definite.
The zero vector and the $i$-th standard basis vector
in $\setR[n]$ are denoted by $0_n$ and $e_i$ respectively,
the $m\times n$ zero matrix
by $0_{m\times n}$, and the $n\times n$ identity matrix by $I_n$.
For a vector $x$ and a symmetric positive definite matrix $M$, define
$\|x\|_M=\sqrt{x^\top Mx}$. If $M$ is omitted, $\|x\|$ denotes the Euclidean norm.
The Cholesky factorization of $M$ is denoted by
$M=(M^{1/2})^\top M^{1/2}$. $\otimes$ denotes the Kronecker product. $\|\cdot\|_F$ denotes the Frobenius norm.
$\Ball(x,r)$ denotes the closed ball centered at $x$ with radius $r$.
Let $\mu_X$ denote the normalized Lebesgue probability measure on the compact set $X$. We write $L^2(X,\mathbb{R}) := L^2_{\mu_X}(X,\mathbb{R})$ for the space of $\mu_X$-square-integrable real-valued functions on $X$.

\section{Preliminaries}\label{section:preliminary}

\subsection{Problem Setup}\label{sec:problem_setup}
We consider the unknown nonlinear control system
\begin{equation}
  \dot{x}(t) = \freal(x(t)) + \greal(x(t))u(t).
  \label{eq:true_system_continuous}
\end{equation}
Here, $x(t) \in\Xspace\subset\setR[\nx]$ is the state at time $t \ge 0$,
$u : [0, \infty) \to \Uspace\subset\setR[\nuinput]$ is the control input,
$\Xspace$ and $\Uspace$ are compact sets.
We assume that $e_i \in \Uspace$ for all $i\in\integerset{1,\nuinput}$.
$\freal:\setR[\nx]\to\setR[\nx]$ and $\greal:\setR[\nx]\to\setR[\nx\times\nuinput]$ are unknown functions.
For an initial condition $x(0)=\hat{\x}\in\Xspace$ and a constant input
$u(\tau) = \uinput \in \Uspace$ over $\tau \in [0,t]$, we denote the solution of~\eqref{eq:true_system_continuous},
provided it exists, at time $t\geq0$ by $x(t;\hat{\x},\uinput)$.

We consider the continuous-time system~\eqref{eq:true_system_continuous} under zero-order-hold input with a sampling time $\Ts> 0$. For every integer $k\ge 0$, we define the discrete-time state $\x(k) = x(k\Ts)$ to denote the sampled state and define $\uinput(k)$ by $u(t)=\uinput(k)$ for $t\in[k\Ts,(k+1)\Ts)$.

Our control objective is to drive the sampled state $\x(k)$ to the origin $0_{\nx}$ while satisfying
\[
  (\x(k),\uinput(k))\in\Zsafe,\quad k = 0,1,\ldots,
\]
where $\Zsafe\subseteq\Xspace\times\Uspace$ is the compact constraint set defined by
\[
  \Zsafe=\{(\x,\uinput)\in\setR[\nx]\times\setR[\nuinput]\mid
  h_j(\x,\uinput)\le 0,\ j\in\integerset{1,n_h}\},
\]
where each function $h_j:\setR[\nx]\times\setR[\nuinput]\to\setR$,
$j\in\integerset{1,n_h}$, is continuously differentiable.
We assume that the state $\x$ is measurable,
$
  h_j(0_{\nx},0_{\nuinput}) < 0\quad \forall j\in\integerset{1,n_h},
$
and $\freal(0_{\nx})=0_{\nx}$, i.e., the origin is an equilibrium under the
zero input $u=0_{\nuinput}$.
Even if the latter equilibrium normalization is not satisfied in the original coordinates,
it can often be satisfied by a simple change of coordinates, as explained in
the following remark.
Based on measured data, we learn a Koopman realization and design a
controller that enforces the nonlinear constraints $\Zsafe$ at all sampling instants
and steers the sampled state $\x(k)$ to $0_{\nx}$.

\begin{remark}
  Suppose that the original dynamics admit an equilibrium pair
  $(x_e,u_e)$ satisfying
  $\freal^{\rm org}(x_e)+\greal^{\rm org}(x_e)u_e=0_{\nx}$.
  Then, by introducing the deviation variables
  $\tilde{x}=x^{\rm org}-x_e$ and
  $\tilde{u}=u^{\rm org}-u_e$, the equilibrium pair
  $(x_e,u_e)$ is mapped to $(0_{\nx},0_{\nuinput})$.
  In these coordinates, the resulting dynamics satisfy
  $\freal(0_{\nx})=0_{\nx}$.
\end{remark}

\subsection{SafEDMD}\label{sec:safedmd}

\subsubsection*{Data-driven Bilinear Koopman Realization}

The Koopman operator $\Koopman_t^\uinput$ corresponding to~\eqref{eq:true_system_continuous} under the constant input $u(\tau) = \uinput \in \Uspace$ over $\tau \in [0,t]$ is defined as
\begin{equation}
  (\Koopman_t^\uinput \phi)(\hat{\x}) = \phi(x(t;\hat{\x},\uinput))
\end{equation}
for all $\hat{\x} \in \Xspace$, $\phi \in L^2(\Xspace, \setR)$, and $t \ge 0$. The real-valued functions $\phi$ are called observables.
For a real-vector-valued function $\liftBig: \Xspace \to \setR[\nxlift+1]$, we use the notation $\Koopman_t^\uinput \liftBig$ to denote the vector-valued function defined as
\begin{equation}
  (\Koopman_t^\uinput \liftBig)(\hat{\x}) = [(\Koopman_t^\uinput \phi_1)(\hat{\x}) \ \cdots \  (\Koopman_t^\uinput \phi_{\nxlift+1})(\hat{\x})]^\top
\end{equation}
for all $\hat{\x} \in \Xspace$, where $\phi_\ell$ is the $\ell$-th component of $\liftBig$ for all $\ell \in \integerset{1,\nxlift+1}$.
The Koopman operator $\Koopman_t^\uinput$ is an infinite-dimensional linear operator that describes the evolution of the observables.

As shown in previous works~\cite{POR20, BGSW25, Philipp+2025}, the Koopman operator $\Koopman_t^\uinput$ approximately inherits the control-affine structure, i.e.,

\begin{equation}
  \Koopman_t^\uinput \approx \Koopman_t^0 + \sum_{i=1}^{\nuinput} (\Koopman_t^{e_i} - \Koopman_t^0) \uinput_i
\end{equation}
holds, where $\Koopman_t^0$ and $\Koopman_t^{e_i},\ i \in \integerset{1,\nuinput}$ are the Koopman operators corresponding to the constant inputs $0_{\nuinput}$ and $e_i$, respectively, and $e_i$ is the $i$-th unit vector of $\setR[\nuinput]$.
To determine data-driven estimates of $\Koopman_{T_s}^\uinput$ for a sampling time $T_s > 0$ with proportional error bounds, we apply SafEDMD~\cite{Strasser+26}.

\subsubsection*{Basic SafEDMD Formulation}

First, we choose the dictionary of observables
\begin{equation}
  \liftBig(\x) = [1\ \x^\top\ \phi_{\nx+1}(\x) \ \cdots \  \phi_{\nxlift}(\x)]^\top, \label{eq:lift_function}
\end{equation}
where $\phi_{\ell}:\setR[\nx]\to\setR$, $\ell \in \integerset{\nx+1,\nxlift}$, are continuously differentiable functions satisfying $\phi_{\ell}(0_{\nx})=0$ for all $\ell \in \integerset{\nx+1,\nxlift}$, and $\nxlift > \nx$ is the dimension of the lifted space excluding the constant observable.
This dictionary contains the original state coordinates, a property that will be used later to reproject lifted predictions onto the original state space.
Next, we collect a data set $\dataset = \{\x_j^{\tilu} , \y_j^{\tilu}\}_{j=1}^{\K}$, where $\y_j^{\tilu} = x(T_s; \x_j^{\tilu}, \tilu)$ for $\tilu \in \{0_{\nuinput}, e_1, \ \cdots \ , e_{\nuinput}\}$.
For each $\tilu\in\{0_{n_u},e_1,\ldots,e_{n_u}\}$, the initial states $\x_1^{\tilu},\ldots,\x_{\K}^{\tilu}$ are drawn i.i.d.\ according to $\mu_{\Xspace}$, independently across $\tilu$.

Using the facts that $\liftBig$ contains the constant observable $1$ and $\liftBig(0_{\nx}) = (\Koopman_{T_s}^0 \liftBig)(0_{\nx})$ holds because $\freal(0_{\nx})=0_{\nx}$, we separate the constant and non-constant components to obtain the following structure of the Koopman operator acting on the chosen observables $\liftBig$:
\begin{equation*}
  \Koopman_{T_s}^0 = \begin{bmatrix}
    1 & 0                       \\
    0 & (\Koopman_{T_s}^0)_{22}
  \end{bmatrix},\quad
  \Koopman_{T_s}^{e_i} = \begin{bmatrix}
    1                           & 0                           \\
    (\Koopman_{T_s}^{e_i})_{21} & (\Koopman_{T_s}^{e_i})_{22}
  \end{bmatrix}.
\end{equation*}
Then, we build the finite-dimensional data-driven approximations
\begin{equation*}
  \Koopman_{T_s}^0 = \begin{bmatrix}
    1 & 0_{\nxlift}^\top \\ 0_{\nxlift} & A
  \end{bmatrix}, \quad
  \Koopman_{T_s}^{e_i} = \begin{bmatrix}
    1   & 0_{\nxlift}^\top \\
    b_i & B_i
  \end{bmatrix}.
\end{equation*}
Here, $A \in \setR[\nxlift\times\nxlift]$, $B_i \in \setR[\nxlift\times\nxlift]$, and $b_i \in \setR[\nxlift]$, $i \in \integerset{1, \nuinput}$, are obtained by solving the following optimization problems:

\begin{equation}\label{eq:koopman_optimization}
  \begin{aligned}
     & \underset{A\in \setR[\nxlift\times\nxlift]}{\minimize}  || Y^0 - A X^0 ||_F,                                                       \\
     & \underset{\substack{b_i \in \setR[\nxlift]                                   \\ B_i \in \setR[\nxlift \times \nxlift]}}{\minimize}
    || Y^{e_i} - [b_i\ B_i] X^{e_i} ||_F, \quad i \in \integerset{1,\nuinput},
  \end{aligned}
\end{equation}
where
\begin{equation*}
  \begin{aligned}
    X^0       & = [0_{\nxlift}\ I_{\nxlift}]
                  [\liftBig(\x_1^0)\cdots\liftBig(\x_{\K}^0)],               \\
    X^{e_i}   & = [\liftBig(\x_1^{e_i})\cdots\liftBig(\x_{\K}^{e_i})],       \\
    Y^{\tilu} & = [0_{\nxlift}\ I_{\nxlift}]
                  [\liftBig(\y_1^{\tilu})\cdots\liftBig(\y_{\K}^{\tilu})],   \\
    \tilu     & \in \{0_{\nuinput}, e_1, \ \cdots \ , e_{\nuinput}\}. \notag
  \end{aligned}
\end{equation*}
Based on this formulation, we obtain the following one-step lifted representation:
\begin{equation}\label{eq:koopman_lifted}
  \lift(\xnext) = A\lift(\x) + B_0 \uinput + \sum_{i=1}^{\nuinput} (B_i - A) \uinput_i \lift(\x) + \xi(\x, \uinput),
\end{equation}
where $\lift(\x) = [0_{\nxlift}\ I_{\nxlift}]\liftBig(\x)$, $B_0 = [b_1\ \cdots \ b_{\nuinput}]$,
\begin{equation}
  \truesystem = x(T_s; \x, \uinput)
  \label{eq:true_system}
\end{equation}
and $\xi(\x, \uinput)$ is the approximation error.
The bound below is certified for lifted states generated by original states through the chosen dictionary.
In the following, we use $\truesystem$ as the discrete-time true system obtained by sampling the continuous-time system~\eqref{eq:true_system_continuous} with the sampling time $T_s$.

\subsubsection*{Proportional Error Bound for SafEDMD}
We define the dictionary $\dict \coloneq \{\liftBig_\ell\}_{\ell=1}^{\nxlift+1}$ representing the $(\nxlift+1)$-dimensional subspace spanned by the chosen observables $\liftBig$.
The following lemma provides a proportional error bound for the SafEDMD approximation of the Koopman operator, which is crucial for the subsequent robust control design.
\begin{lemma}[{\cite[Cor.~3.2.]{Strasser+26}}]\label{lemma:proportional_error_bound}
  If there exists a proportional error bound on the projection error, i.e.,
  \begin{equation}\label{eq:projection_error_bound}
    \|(\Koopman_{T_s}^\uinput \liftBig)(\x) - (P_{\dict} \Koopman_{T_s}^\uinput|_{\dict} \liftBig)(\x) \| \le  \tilde{c}_x\|\lift(\x)\| + \tilde{c}_u\|\uinput\|,
  \end{equation}
  then for any probabilistic tolerance $\beta \in (0,1)$, amount of data $\K_0 \in \mathbb{N}$, and sampling time $T_s > 0$, there exist constants $\bar{c}_x, \bar{c}_u \in \mathcal{O}(1/\sqrt{\beta \K_0} + T_s^2)$ such that
  \begin{equation}\label{eq:error_bound_proportional}
    \|\xi(\x, \uinput)\| \le \cx\|\lift(\x)\| + \cu\|\uinput\|
  \end{equation}
  holds for all $\K \geq \K_0, \x \in \Xspace$, and $\uinput \in \Uspace$ with probability $1-\beta$,
  where $\cx = \tilde{c}_x + \bar{c}_x, \cu = \tilde{c}_u + \bar{c}_u$,
  the probability is defined with respect to the product probability measure $\mu_{\Xspace}^{\otimes \K(1+\nuinput)}$ and $P_{\dict}$ is the orthogonal projection onto the subspace spanned by $\dict$.
\end{lemma}

\begin{remark}
  The proportional error bound on the projection error~\eqref{eq:projection_error_bound} is consistent with recent finite-data residual error bounds for Koopman-based approximations.
  In particular, uniform projection-error bounds can be derived from polynomial tests, and interpolation arguments can be used to obtain proportional bounds; see, e.g.,~\cite{ITS24,Kohne+25,YM25}.
\end{remark}


\subsection{Tube-Based MPC}

\subsubsection*{Basic Robust MPC Formulation}
In real-world applications, the prediction equation often has uncertainties due to approximation errors, disturbances, and noise.
In such cases, the true system can be represented by discrete-time dynamics with uncertainties:
\begin{equation}
  \xnext = \fW(\x, \uinput, \d), \d\in\D,
  \label{eq:disturbed_system_mpc}
\end{equation}
where $\d\in\D$ is an unknown uncertainty, and $\D$ is a compact set that bounds the possible uncertainties.
MPC that guarantees constraint satisfaction under prediction
uncertainties as in~\cref{eq:disturbed_system_mpc} is called robust MPC
(RMPC). RMPC at time $k$ proceeds as follows:
\rmpcenumerate{\fW}{eq:disturbed_system_mpc}
Here, $\x(k)$ is the observed state of the true system at time $k$.

At each control step, RMPC predicts future trajectories over the horizon $\N$ and
solves an optimization problem against worst-case uncertainties to find a robust optimal control policy.
However, solving this min-max problem directly is often computationally expensive~\cite{LCRM04}.

\subsubsection*{Basic Tube-based MPC Formulation}
A practical alternative is tube-based MPC (TMPC).
TMPC first defines a nominal trajectory without uncertainties as
$\nox_{k+1}=\fW(\nox_k,\nou_k,0)$ for all $k\ge0$.
Here, $\nox_k\in\Xspace$ and $\nou_k\in\Uspace$ are the nominal state and input.
Then, a tube $\tube_{\nox_k}$ and a feedback law
$\kappa:\Xspace\times\Zsafe\to\Uspace$ are designed so that all possible
closed-loop true states remain inside the tube. The tube and feedback law are designed to satisfy
\tubefeedbackcond{\fW}
Thus, if the initial state belongs to $\tube_{\nox_0}$, all possible states controlled by $\kappa$
remain in $\tube_{\nox_k}$ for all $k\ge0$.

Next, nominal constraints ${\Zsafe}'_k$ are imposed so that
\nominalconst{useZ}
holds,
which guarantees
\tmpcprop{useZ}
In other words, if the initial state belongs to the initial tube and the
nominal trajectories satisfy tightened nominal constraints, then all
true trajectories controlled
by the feedback law satisfy the original constraints.
Based on this, TMPC runs as follows:
\tmpcenumerate{eq:disturbed_system_mpc}{\fW}{useZ}
The resulting policy $\pi(\x(k'),k')=\kappa(\x(k'),\noxt{k'-k},\nout{k'-k})$
is feasible for the original RMPC problem while reducing the computational burden
compared with direct min-max RMPC\@.

\section{Proposed Method}\label{section:proposed_method}

This section derives the three main ingredients of the proposed framework: an error-aware discrete-time predictor based on~\cref{sec:safedmd}, a discrete-time RCCM for this predictor, and a discrete-time RCCM-based TMPC formulation.
The construction builds on CCM-based TMPC studies~\cite{SZK23,GSS24,ZS24} and is tailored to the proportional Koopman approximation-error bound introduced in~\cref{sec:safedmd}.

\subsection{Error-Aware Predictor}

We first examine the multi-step prediction issue for bilinear Koopman realizations in more detail.
The discussion below closely follows Section~3 of existing work~\cite{GMSW25}; we adapt it to the family of Koopman operators associated with constant control inputs.

We define the $\nx$-dimensional embedded submanifold induced by a lifting function $\lift$ as
\begin{equation*}
  \mathcal M = \{ \lift(\x) \mid \x \in \setR[\nx] \} \subseteq \setR[\nxlift].
\end{equation*}
This manifold $\mathcal M$ is invariant with respect to the true Koopman operator $\Koopman_t^\uinput$, that is, for any $t\ge 0,  \uinput\in\Uspace$, it holds that
\begin{equation*}
  ( \Koopman_t^\uinput \lift ) (\hat \x) = \lift(x(t; \hat \x, \uinput)) \in \mathcal M, \quad \forall \hat \x \in \Xspace.
\end{equation*}
This means that the true Koopman operator maps an $\mathcal M$-valued function to an $\mathcal M$-valued function.
But for data-driven Koopman realizations, the invariance of the manifold $\mathcal M$ is not guaranteed.
For linear Koopman realizations, if predicted lifted states are evaluated at points outside of $\mathcal M$,
we can establish the error bound by exploiting their linear structure~\cite{Zhang+22,MCV22,KTZS25}.
But for bilinear Koopman realizations, this lack of invariance can lead to the evaluation of the learned Koopman realization at points outside of $\mathcal M$, which can result in intractable errors in the prediction.
To preserve this property for $\dict \coloneq \{\lift_\ell\}_{\ell=1}^{\nxlift+1}$, Koopman dictionary invariance~\cite{GP22}, i.e., $\Koopman_t^\uinput \dict \subseteq \dict$, is required for all $t\ge 0, \uinput\in\Uspace$.
If this assumption is satisfied for the true Koopman operator, we can estimate the error bound of the multi-step prediction even when we use the data-driven Koopman realization~\cite{Worthmann+24,BGSW25}. However, this assumption is difficult to satisfy in practice.
In this work, we avoid this issue by projecting the predicted lifted states onto the original state coordinates before the next step of prediction, which ensures that the predicted lifted states are always evaluated at points in $\mathcal M$.

\subsubsection*{Error-Aware Predictor Formulation based on SafEDMD}

Because the proportional error bound in~\cref{lemma:proportional_error_bound} is certified for lifted states generated from original states, we use the original state coordinates contained in the dictionary to define the MPC predictor in the original state space.
When the proportional error bounds~\eqref{eq:error_bound_proportional} hold, based on~\eqref{eq:koopman_lifted} and~\eqref{eq:true_system}, the one-step map of the true system can be written with an unknown $\mathbf{e} \in \Ball(0_{\nx}, 1)$ as
\begin{equation}
  \xnext=
  \Fapprox(\x,\uinput)+( \cx\|\lift(\x)\| + \cu\|\uinput\| )\mathbf{e},
  \label{eq:uncertain_system_nonsmooth}
\end{equation}
where $(\x, \uinput)\in\Zsafe$, $\xnext = \Freal(\x, \uinput)$, $\cx$ and $\cu$ are the proportional error bound constants defined in~\cref{lemma:proportional_error_bound}, and
$\Fapprox$ is defined as
\begin{equation}
  \begin{aligned}
      & \Fapprox(\x, \uinput)                     \\
    = & [ I_{\nx} 0_{\nx \times (\nxlift - \nx)}]
    (A\lift(\x) + B_0 \uinput + \sum_{i=1}^{\nuinput} (B_i - A) \uinput_i \lift(\x)) \label{eq:approx_system}.
  \end{aligned}
\end{equation}
The map $\Fapprox$ first propagates the lifted state by the learned Koopman realization and then projects the result onto the original state coordinates.
This means that the one-step true-system map can be written as the nominal predictor $\Fapprox$ plus an unknown uncertainty $(\cx\|\lift(\x)\| + \cu\|\uinput\|)\mathbf{e}$.
The control-affine structure of the nominal predictor is crucial for the subsequent analysis because it allows us to apply the discrete-time RCCM-based control design method.
Furthermore, due to the SafEDMD formulation and the choice of observables satisfying $\lift(0_{\nx})=0_{\nxlift}$, the approximation-error term vanishes as $\x$ and $\uinput$ approach the target equilibrium. This property is also crucial for the subsequent analysis because it allows us to guarantee convergence of the true state.


\subsubsection*{Differentiable Uncertainty Description}

Next, we introduce a differentiable uncertainty description so that the derivatives required for the discrete-time RCCM-based control design are well-defined.
The error term $( \cx\|\lift(\x)\| + \cu\|\uinput\| )\mathbf{e}$ is not differentiable at points where $\lift(\x) = 0_{\nxlift}$ or $\uinput = 0_{\nuinput}$ because of the norm terms.

To address this issue, we introduce an equivalent uncertainty description as follows.

\begin{lemma}\label{lemma:equivalent_uncertainty}
  Define $B_1 = \{ (\cx\|\lift(\x)\| + \cu\|\uinput\|) \mathbf{e} \mid (\x, \uinput) \in \Zsafe, \mathbf{e} \in \Ball(0_{\nx}, 1)\}$ and $B_2 = \{ \Eapprox(\x,\uinput) \d \mid (\x, \uinput) \in \Zsafe,\ \d \in \D \}$, where $\Eapprox(\x, \uinput) = [ \cx( \lift(\x)^\top \otimes I_{\nx} ) \ \cu (\uinput^\top \otimes I_{\nx}) ]$ and $\D = \{\begin{bmatrix}
      \d_{\lift} \otimes \d_{\x} \\
      \d_{\uinput} \otimes \d_{\x}\end{bmatrix} \mid \d_{\lift} \in \Ball(0_{\nxlift}, 1), \d_{\uinput} \in \Ball(0_{\nuinput}, 1), \d_{\x} \in \Ball(0_{\nx}, 1)\}$.
  Then, it holds that $B_1 = B_2$.
\end{lemma}

\begin{proof}
  First, we show that $B_1 \subseteq B_2$. For any $\y_1 \in B_1$, there exist $(\x, \uinput) \in \Zsafe$ and $\mathbf{e} \in \Ball(0_{\nx}, 1)$ such that $\y_1 = ( \cx\|\lift(\x)\| + \cu\|\uinput\| )\mathbf{e}$. Define $\d_{\lift} = \lift(\x)/\|\lift(\x)\|$ if $\lift(\x)\ne 0_{\nxlift}$ and choose any $\d_{\lift}\in\Ball(0_{\nxlift},1)$ otherwise. Similarly, define $\d_{\uinput} = \uinput/\|\uinput\|$ if $\uinput\ne 0_{\nuinput}$ and choose any $\d_{\uinput}\in\Ball(0_{\nuinput},1)$ otherwise. Let $\d_{\x} = \mathbf{e}$, and $\d = \begin{bmatrix}
      \d_{\lift} \otimes \d_{\x} \\
      \d_{\uinput} \otimes \d_{\x}\end{bmatrix}$. Then, we have $\d \in \D$ and
  \begin{align*}
    \Eapprox(\x, \uinput) \d & = [ \cx( \lift(\x)^\top \otimes I_{\nx} ) \ \cu (\uinput^\top \otimes I_{\nx}) ] \begin{bmatrix}
                                                                                                                  \d_{\lift} \otimes \d_{\x} \\
                                                                                                                  \d_{\uinput} \otimes \d_{\x}\end{bmatrix} \\
                             & = \cx (\lift(\x)^\top \d_{\lift}) \d_{\x} + \cu (\uinput^\top \d_{\uinput}) \d_{\x}                                          \\
                             & = (\cx\|\lift(\x)\| + \cu\|\uinput\|)\mathbf{e} = \y_1.
  \end{align*}
  Thus, $B_1 \subseteq B_2$.
  Next, we show that $B_2 \subseteq B_1$. For any $\y_2 \in B_2$, there exist $(\x, \uinput) \in \Zsafe$ and $\d = \begin{bmatrix}
      \d_{\lift} \otimes \d_{\x} \\
      \d_{\uinput} \otimes \d_{\x}\end{bmatrix} \in \D$ such that $\y_2 = \Eapprox(\x, \uinput) \d$. Then, we have
  \begin{align}
    \|\y_2\| & = \| \cx (\lift(\x)^\top \d_{\lift}) \d_{\x} + \cu (\uinput^\top \d_{\uinput}) \d_{\x} \| \nonumber       \\
             & \le \cx \|\lift(\x)\| \|\d_{\x}\| \|\d_{\lift}\| + \cu \|\uinput\| \|\d_{\x}\| \|\d_{\uinput}\| \nonumber \\
             & \le \cx\|\lift(\x)\| + \cu\|\uinput\|\;(\because \|\d_{\lift}\|, \|\d_{\uinput}\|, \|\d_{\x}\| \le 1).
  \end{align}
  If $\cx\|\lift(\x)\|+\cu\|\uinput\|>0$, this inequality implies $\y_2=(\cx\|\lift(\x)\|+\cu\|\uinput\|)\mathbf{e}$ for some $\mathbf{e}\in\Ball(0_{\nx},1)$; if $\cx\|\lift(\x)\|+\cu\|\uinput\|=0$, then $\y_2=0_{\nx}$ and the same conclusion holds.
  Therefore, $B_2 \subseteq B_1$ holds. Combining $B_1 \subseteq B_2$ and $B_2 \subseteq B_1$, we conclude that $B_1 = B_2$.
\end{proof}
Based on the above lemma, we can express~\eqref{eq:uncertain_system_nonsmooth} as the following differentiable error-aware predictor:

\begin{equation}
  \xnext= \fW(\x,\uinput,\d) =
  \Fapprox(\x,\uinput)+\Eapprox(\x,\uinput)\d,
  \label{eq:uncertain_system}
\end{equation}
where $\Fapprox$ is defined in~\eqref{eq:approx_system}, $\Eapprox(\x, \uinput)$ and $\D$ are defined in~\cref{lemma:equivalent_uncertainty}, and $\d \in \D$ is the unknown uncertainty.

\subsection{Discrete-Time Robust Control Contraction Metric}

We introduce a discrete-time robust control contraction metric for our error-aware predictor.
Control contraction metrics~\cite{MS15} are a powerful tool for nonlinear control design, and their robust extensions~\cite{Zhao+22} have been developed in recent years and also applied to discrete-time settings~\cite{GSS24, ZS24}.
It serves as the main tool for robust control design in this work.
\begin{assumption}\label{assump:drccm}
  There exist a continuously differentiable matrix function $M:\setR[\nx]\to\mathbb S_+^{\nx}$,
  a continuous function $K:\setR[\nx]\to\setR[\nuinput \times \nx]$, and constants
  $0<\rho_c<1$, $\alpha_1>0$, $\alpha_2>0$ such that for all
  $(\x,\uinput)\in\Zsafe$, $\d\in\D$,
  \begin{subequations}\label{eq:drccm_conditions_full}
    \begin{align}
      \Acl(\x,\uinput,\d)^\top M(\xnext)\Acl(\x,\uinput,\d)
                              & \preceq (1-\rho_c)M(\x), \label{eq:drccm_contraction}
      \\
      \alpha_1 I\preceq M(\x) & \preceq \alpha_2 I,
      \label{eq:drccm_conditions}                                                     \\
      M(\xnext)               & \preceq \alpha_2 I,
      \label{eq:drccm_conditions2}
    \end{align}
  \end{subequations}
  where $\Acl(\x,\uinput,\d)=\left.\frac{\partial\fW}{\partial\x}\right|_{(\x,\uinput,\d)}
    +\left.\frac{\partial\fW}{\partial\uinput}\right|_{(\x,\uinput,\d)}K(\x)$ and $\xnext=\fW(\x,\uinput,\d)$.
\end{assumption}

\begin{remark}\label{rem:sos}
  With the change of variables $W=M^{-1}, Y = KW$,
  \cref{eq:drccm_conditions_full} can be transformed into linear matrix
  inequalities and solved numerically~\cite{Zhao+22,WMB21,MS15}.
  In particular, if the Koopman observables are chosen
  as polynomials in~\eqref{eq:lift_function}, then $\fW$ and $\Acl$ become polynomial functions of the state
  and input. By restricting $W$ and $Y$ to polynomial functions, the above
  conditions can be relaxed into a sum-of-squares formulation.
\end{remark}

%
\begin{remark}
  In~\cref{eq:drccm_contraction}, we need the derivatives of the error-aware predictor $\fW$ at points in $\Zsafe$.
  By the problem setup in~\cref{sec:problem_setup}, $\Zsafe$ contains the origin $(0_{\nx},0_{\nuinput})$. Therefore, if we used the nonsmooth uncertainty description in~\cref{eq:uncertain_system_nonsmooth}, the derivatives at the origin would not be well-defined. To avoid this issue, we use the equivalent differentiable error-aware predictor in~\cref{eq:uncertain_system} for control design.
\end{remark}

\subsection{Tube and Feedback Design}

Next, we design the tube and feedback law.
For any $\x,\nox\in\setR[\nx]$, let $\Gamma(\nox,\x)$ be the set of
component-wise continuously differentiable curves $\gamma:[0,1]\to\setR[\nx]$ satisfying
$\gamma(0)=\nox$, $\gamma(1)=\x$.
Given a discrete-time RCCM satisfying~\cref{eq:drccm_conditions_full}, define the associated Riemannian distance $V$ as
\begin{equation}
  \riemannenergyimpl,
\end{equation}
and denote a minimizer by the geodesic $\gamma^*$, which exists and is unique almost everywhere under
\cref{eq:drccm_conditions}~\cite{MS15}.
For any $\x,\nox\in\setR[\nx]$, set
\begin{align}
  \gamma^u(s)          & =\nou+\int_0^s K(\gamma^*(s'))\dot{\gamma}^*(s')\,ds' \\
  \kappa(\x,\nox,\nou) & =\gamma^u(1).
\end{align}

\begin{proposition}\label{prop:tube_dynamics}
  Suppose \cref{assump:drccm} holds and the proportional error bound~\cref{eq:error_bound_proportional} is valid. Then, for any
  $\x,\nox\in\setR[\nx],\nou \in \setR[\nuinput]$
  satisfying
  $(\gamma^*(s),\gamma^u(s))\in\Zsafe$ for all $s\in[0,1]$, it holds that
  \[
    V(\xnext,\noxnext)\le\sqrt{1-\rho_c}\,V(\x,\nox)+\sqrt{\alpha_2}( \cx\|\lift(\nox)\| + \cu\|\nou\| )
  \]
  where
  $\xnext=\Freal(\x,\uinput)$,
  $\noxnext=\Fapprox(\nox,\nou)$,
  $\uinput=\kappa(\x,\nox,\nou)$.
\end{proposition}

\begin{proof}
  There exists $\d\in\D$ such that
  $\xnext=\fW(\x,\uinput,\d)$ by the definition of $\D$.
  Define $\y_+ = \fW(\nox, \nou, \d)$ and $c_1^+(s) = \fW(\gamma^*(s), \gamma^u(s), \d)$.
  Then $c_1^+(0) = \y_+$ and $c_1^+(1) = \xnext$, so $c_1^+ \in \Gamma(\y_+, \xnext)$.
  Moreover, $\dot c_1^+(s) = \Acl(\gamma^*(s), \gamma^u(s),\d)\dot{\gamma}^*(s)$.
  By~\cref{assump:drccm}, we have
  \[\|\dot c_1^+(s)\|_{M(c_1^+(s))} \leq \sqrt{1-\rho_c}\,\|\dot\gamma^*(s)\|_{M(\gamma^*(s))}. \]
  Integrating both sides over $s\in[0,1]$ yields
  \[ V(\xnext,\y_+) \leq \int_0^1 \|\dot c_1^+(s)\|_{M(c_1^+(s))} ds \leq \sqrt{1-\rho_c}\,V(\x,\nox). \]
  Next, define $c_2^+(s)=\fW( \nox, \nou, s\d )$.
  Then $c_2^+(0)=\noxnext$ and $c_2^+(1)=\fW(\nox,\nou,\d)=\y_+$, so $c_2^+\in\Gamma(\noxnext,\y_+)$ and $\dot c_2^+(s)=\Eapprox(\nox,\nou)\d$.
  From \cref{lemma:equivalent_uncertainty}, $\|\Eapprox(\nox,\nou)\d\|\le( \cx\|\lift(\nox)\| + \cu\|\nou\| )$.
  And, for all $s \in [0,1]$,  $s \d \in \D$ by the construction of $\D$ in \cref{lemma:equivalent_uncertainty}, so using \cref{eq:drccm_conditions2}, we have $M(c_2^+(s)) \preceq \alpha_2 I$.
  Therefore, we obtain
  \begin{align*}
    \|\dot c_2^+(s)\|_{M(c_2^+(s))}\leq & \sqrt{\alpha_2}\|\Eapprox(\nox,\nou)\d\|             \\
    \leq                                & \sqrt{\alpha_2}( \cx\|\lift(\nox)\| + \cu\|\nou\| ).
  \end{align*}
  Integrating both sides over $s\in[0,1]$ yields
  \begin{align*}
    V(\y_+,\noxnext) & \leq \int_0^1 \|\dot c_2^+(s)\|_{M(c_2^+(s))} ds          \\
                     & \leq \sqrt{\alpha_2}( \cx\|\lift(\nox)\| + \cu\|\nou\| ).
  \end{align*}
  Therefore,
  \begin{align*}
    V(\xnext,\noxnext) & \leq  V(\xnext,\y_+) + V(\y_+,\noxnext)                                                \\
                       & \leq  \sqrt{1-\rho_c}\,V(\x,\nox)+\sqrt{\alpha_2}( \cx\|\lift(\nox)\| + \cu\|\nou\| ).
  \end{align*}
\end{proof}

\begin{remark}\label{remark:tube_dynamics}
  \Cref{prop:tube_dynamics} implies that the value of $V$ between the next true state $\xnext$ and the next nominal state $\noxnext$ can be bounded by a contraction term $\sqrt{1-\rho_c}\,V(\x,\nox)$ and an error term $\sqrt{\alpha_2}( \cx\|\lift(\nox)\| + \cu\|\nou\| )$.
  The key point is that the error term vanishes as the nominal state and input go to zero, so only the contraction term remains. This property is useful for guaranteeing convergence of the true state.
\end{remark}

The geodesic can be computed numerically via a Chebyshev pseudospectral method with proper discretization~\cite{LM17}.

\subsection{Nominal Constraint Design}

We next tighten the constraints for the nominal predictor.

\begin{proposition}[{\cite[Prop.~5]{SZK23}}]\label{prop:nominal_constraint}
  Under~\cref{assump:drccm}, for any
  $\x,\nox\in\setR[\nx]$ and $\nou \in \setR[\nuinput]$
  satisfying
  \begin{equation}
    \begin{aligned}
       & h_j(\nox,\nou)+c_jV(\x,\nox)\le0,\ \forall j\in\integerset{1,n_h},                                       \\
       & c_j=\max_{(\nox,\nou)\in\Zsafe}\left\|\left(\left.\frac{\partial h_j}{\partial \x}\right|_{(\nox,\nou)}+
      \left.\frac{\partial h_j}{\partial \nou}\right|_{(\nox,\nou)}K(\nox)\right)M(\nox)^{-\frac{1}{2}}\right\|,
    \end{aligned}
  \end{equation}
  it holds that
  \begin{equation}
    (\gamma^*(s),\gamma^u(s))\in\Zsafe, \ s\in[0,1].
  \end{equation}
\end{proposition}

Using this proposition, we obtain the following theorem.

\begin{theorem}\label{theorem:tube_feedback}
  Suppose \cref{assump:drccm} holds and the proportional error bound~\cref{eq:error_bound_proportional} is valid. Consider an initial state $\x(0)$ and sequences
  $\nox_k,\nou_k,\delta_k$ satisfying
  \begin{subequations}
    \begin{align}
       & h_j(\nox_k,\nou_k)+c_j\delta_k\le0,
      \quad \forall j\in\integerset{1,n_h},
      \label{eq:nominal_constraint}                                                                    \\
       & \nox_{k+1}=\Fapprox(\nox_k,\nou_k),
      \label{eq:nominal_dynamics}                                                                      \\
       & V(\x(0),\nox_0)\le\delta_0,
      \label{eq:initial_tube}                                                                          \\
       & \delta_{k+1}=\sqrt{1-\rho_c}\delta_k+\sqrt{\alpha_2}( \cx\|\lift(\nox_k)\| + \cu\|\nou_k\| ).
      \label{eq:tube_dynamics}
    \end{align}
  \end{subequations}
  Then, for any $k\in\{0,1,\ldots\}$, the true trajectory $\x(k+1)=\Freal(\x(k),\kappa(\x(k),\nox_k,\nou_k))$ satisfies
  \begin{subequations}
    \begin{align}
      V(\x(k),\nox_k)                     & \le\delta_k,
      \label{eq:theorem_tube}                            \\
      (\x(k),\kappa(\x(k),\nox_k,\nou_k)) & \in\Zsafe.
      \label{eq:theorem_constraint}
    \end{align}
  \end{subequations}
\end{theorem}

\begin{proof}
  The proof proceeds by induction on $k$.
  For $k=0$, \cref{eq:theorem_tube} follows from \cref{eq:initial_tube}.
  Also, from \cref{eq:nominal_constraint}, \cref{eq:initial_tube}, and
  \cref{prop:nominal_constraint},
  $(\gamma^*(s),\gamma^u(s))\in\Zsafe$ for all $s\in[0,1]$;
  substituting $s=1$ gives \cref{eq:theorem_constraint}.
  Assume the claim holds at $k=\ell$.
  By \cref{eq:theorem_tube}, \cref{eq:tube_dynamics} at $k=\ell$, and
  \cref{prop:tube_dynamics},
  $V(\x(\ell+1),\nox_{\ell+1})
    \le\sqrt{1-\rho_c}\,V(\x(\ell),\nox_\ell)+\sqrt{\alpha_2}( \cx\|\lift(\nox_\ell)\| + \cu\|\nou_\ell\| )
    \le\sqrt{1-\rho_c}\,\delta_\ell +\sqrt{\alpha_2}( \cx\|\lift(\nox_\ell)\| + \cu\|\nou_\ell\| )=\delta_{\ell+1}$.
  Thus, \cref{eq:theorem_tube} holds at $k=\ell+1$.
  Combining this with \cref{eq:nominal_constraint} at $k=\ell + 1$ and
  \cref{prop:nominal_constraint},
  $(\gamma^*(s),\gamma^u(s))\in\Zsafe$ for all $s\in[0,1]$ at $k=\ell+1$;
  substituting $s=1$ gives \cref{eq:theorem_constraint}.
  Hence both statements hold for all $k\ge0$.
\end{proof}

\subsection{TMPC Formulation}

\subsubsection*{TMPC Optimization Problem}
Combining the above ingredients, we present the TMPC formulation.
Given the state $\x(k)$ at time $k$, we solve the following horizon-$\N$ optimization problem.
\begin{problem}\label{prob:tmpc}
  \begin{subequations}\label{eq:mpc}
    \begin{align}
      \underset{\nox_{\cdot|k},\nou_{\cdot|k},\delta_{\cdot|k}}{\minimize} & \sum_{i=0}^{\N-1}\ell(\nox_{i|k},\nou_{i|k}, \delta_{i|k})
      +\ell_f(\nox_{\N|k}, \delta_{\N|k}) \notag                                                                                                                               \\
      \subjectto\quad
                                                                           & \nox_{i+1|k}=\Fapprox(\nox_{i|k},\nou_{i|k}),
      \label{eq:nox}                                                                                                                                                           \\
                                                                           & \delta_{i+1|k}=\sqrt{1-\rho_c}\,\delta_{i|k}                                               \notag \\
                                                                           & \phantom{\delta_{i+1|k}=}+\sqrt{\alpha_2}( \cx\|\lift(\nox_{i|k})\| + \cu\|\nou_{i|k}\| ),
      \label{eq:delta}                                                                                                                                                         \\
                                                                           & h_j(\nox_{i|k},\nou_{i|k})+c_j\delta_{i|k}\le0,
      \label{eq:constraint}                                                                                                                                                    \\
                                                                           & V(\x(k), \nox_{0|k})=\delta_{0|k},
      \label{eq:tubeinit}                                                                                                                                                      \\
                                                                           & (\nox_{\N|k},\delta_{\N|k})\in\liXf,
      \label{eq:terminal}                                                                                                                                                      \\
                                                                           & \forall i\in\integerset{0,\N-1},\ \forall j\in\integerset{1,n_h} \notag
    \end{align}
  \end{subequations}
\end{problem}
with
\begin{equation}\label{eq:cost}
  \ell(\nox_{i|k},\nou_{i|k},\delta_{i|k})=
  \|\nox_{i|k}\|_Q^2+\|\nou_{i|k}\|_R^2 + \lambda\delta_{i|k}^2,
\end{equation}
where $Q\in \mathbb S_+^{\nx}$ and $R\in \mathbb S_+^{\nuinput}$ are weighting
matrices, and $\lambda\in\setR_{>0}$ is a weighting scalar for the tube radius. Unlike standard TMPC formulations~\cite{LSH19, Gonzalez+11, SZK23}, we introduce the tube-radius cost $\lambda\delta_{i|k}^2$, which is useful for guaranteeing convergence of the true state.
The decision variables are the nominal state and input trajectories
$\nox_{i|k},\nou_{i|k}$ and the tube radius $\delta_{i|k}$. The nominal state $\nox_{i|k}$ satisfies the nominal predictor $\Fapprox$ in~\cref{eq:approx_system} (cf.~\cref{eq:nox}).
The tube radius $\delta_{i|k}$ evolves according to \cref{prop:tube_dynamics} (cf.~\cref{eq:delta}, \cref{eq:tubeinit}). The constraints on the nominal state and input in \cref{eq:constraint} are designed based on \cref{theorem:tube_feedback} to ensure that the sampled true state-input pair satisfies the original nonlinear constraints at each sampling instant.
The terminal constraint \cref{eq:terminal} and the terminal cost $\ell_f$ are introduced to guarantee recursive feasibility and convergence, which we will discuss in \cref{sec:terminal_set}.
If $\{\nox_{i|k}^*,\nou_{i|k}^*,\delta_{i|k}^*\mid i\in\integerset{0,\N}\}$
is the solution at time $k$, the applied input is
\begin{equation}\label{eq:mpc_control_law}
  \uinput(k)=\kappa(\x(k),\nox_{0|k}^*,\nou_{0|k}^*).
\end{equation}

\subsubsection*{Offline and Online Stages}
The proposed controller is implemented in two stages.
First, in the offline stage, the error-aware predictor based on the SafEDMD and all ingredients required for~\cref{prob:tmpc} are constructed from the collected data.
The resulting offline procedure is summarized in~\cref{algorithm:mpc_offline}.

\begin{algorithm}[H]
  \caption{Learning and TMPC Offline Design}
  \label{algorithm:mpc_offline}
  \begin{algorithmic}[1]
    \REQUIRE \dataset = $\{\x_j^{\tilu}, \y_j^{\tilu}\}_{j=1}^\K$ with $\y_j^{\tilu} = \Freal(\x_j^{\tilu}, \tilu)$ for $\tilu \in \{0_{\nuinput}, e_1, \ldots, e_{\nuinput}\}$ and lifting function $\lift$
    \STATE Learn coefficient matrices $A, B_i$ using~\cref{eq:koopman_optimization} and determine $\cx, \cu$ using~\cref{eq:error_bound_proportional} to construct the error-aware predictor $\fW$ in~\cref{eq:uncertain_system}
    \STATE Compute $M(\x),K(\x), \alpha_1, \alpha_2$ and $\rho_c$ (\cref{assump:drccm})
    \STATE Compute $c_j$ (\cref{prop:nominal_constraint})
    \STATE Design $\liXf$, $k_f$, and $\ell_f$ (\cref{ass:terminalset,prop:terminalset})
  \end{algorithmic}
\end{algorithm}

In the online stage, we measure the current state $\x(k)$ and solve~\cref{prob:tmpc}.
A key property of the proposed TMPC formulation is that a feasible solution can be constructed by shifting the previous solution, as shown in the proof of~\cref{thm:tmpc}.

Then, we run the numerical solver for at most the prescribed computation time constraint $\tau$. If we find a solution with a smaller objective value, we update the current best feasible solution.
If we find no improved feasible solution, we use the shifted feasible solution for control.
Thus, we use the online optimization to improve the available feasible solution, while feasibility does not rely on the optimality of~\cref{prob:tmpc} within $\tau$.
Furthermore, if
\begin{equation}\label{eq:local_solution}
  \begin{aligned}
                            & \nox_{i|k}^\text{loc} = 0_{\nx},\, \nou_{i|k}^\text{loc} = 0_{\nuinput}, \\
    \delta_{i|k}^\text{loc} & = ( \sqrt{1-\rho_c} )^iV(\x(k), 0_{\nx}),\, i \in \integerset{0,\N}
  \end{aligned}
\end{equation}
is a feasible solution of~\cref{prob:tmpc} and the cost for this solution is smaller than the current best feasible solution, we update the current best feasible solution with this local solution.
This local solution is required to guarantee exponential stability of the closed-loop system, which we will discuss in~\cref{thm:es}.
The resulting online procedure is summarized in~\cref{algorithm:mpc_online}.

\begin{algorithm}[H]
  \caption{TMPC Online Control}
  \label{algorithm:mpc_online}
  \begin{algorithmic}[1]
    \REQUIRE Weights $Q,R,\lambda$, horizon $\N$, and time constraint $\tau$
    \FOR{each time step $k\ge0$}
    \IF{$k=0$}
    \STATE Set the initial feasible solution as the candidate
    \ELSE
    \STATE Set the shifted previous solution as the candidate
    \ENDIF
    \STATE Run the solver for~\cref{prob:tmpc} within at most $\tau$
    \IF {a feasible solution with smaller cost is found}
    \STATE Update the candidate solution
    \ENDIF
    \IF {\cref{eq:local_solution} is feasible and has smaller cost}
    \STATE Update the candidate solution
    \ENDIF
    \STATE Apply \cref{eq:mpc_control_law}
    \ENDFOR
  \end{algorithmic}
\end{algorithm}

\subsection{Terminal Set Constraint}\label{sec:terminal_set}

To guarantee recursive feasibility and convergence of the nominal trajectories, we introduce a terminal set.

\begin{assumption}\label{ass:terminalset}
  There exist a terminal set
  $\liXf\subseteq\setR[\nx]\times\setR_{\ge0}$,
  a terminal controller $k_f:\setR[\nx]\to\setR[\nuinput]$, and a terminal cost
  $\ell_f:\setR[\nx]\times\setR_{\ge0}\to\setR_{\ge0}$ such that for any
  $(\nox,\delta)\in\liXf$ and for all $j \in \integerset{1,n_h}$,
  \begin{subequations}
    \begin{align}
       & (\noxnext,\deltanext)\in\liXf,
      \label{eq:terminal_set}                                                                                                                               \\
       & h_j(\nox,k_f(\nox))+c_j\delta\le0,
      \ h_j(\noxnext,k_f(\noxnext))+c_j\deltanext\le0,
      \label{eq:terminal_constraint}                                                                                                                        \\
       & \ell(\nox,k_f(\nox),\delta)\le\ell_f(\nox, \delta)-\ell_f(\noxnext, \deltanext),
      \label{eq:terminal_cost}                                                                                                                              \\
       & (\nox, \delta)\in\liXf, \hat \delta \in [0, \delta] \Rightarrow (\nox, \hat \delta)\in\liXf,\, \ell_f(\nox, \hat \delta) \leq \ell_f(\nox, \delta)
      \label{eq:terminal_cost_monotone}
    \end{align}
  \end{subequations}
  hold, where $\noxnext=\Fapprox(\nox,k_f(\nox))$ and
  $\deltanext=\sqrt{1-\rho_c}\delta+\sqrt{\alpha_2}( \cx\|\lift(\nox)\| + \cu\|k_f(\nox)\| )$.

\end{assumption}

The following proposition shows that the terminal ingredients can be explicitly constructed by using the fact that the error term vanishes and only the contraction term remains as the nominal state and input go to zero, as discussed in~\cref{remark:tube_dynamics}.

\begin{proposition}\label{prop:terminalset}
  \cref{ass:terminalset} holds with
  \begin{subequations}
    \begin{align}
      \liXf=\{ & (\nox,\delta)\in\setR[\nx]\times\setR_{\ge0}\mid \notag                                               \\
               & \nox=0_{\nx}, \label{eq:terminal_nox_prop}                                                            \\
               & h_j(\nox,k_f(\nox))+c_j\delta\le0,\quad \forall j\in\integerset{1,n_h} \label{eq:terminal_delta_prop}
      \}
    \end{align}
  \end{subequations}
  $k_f(\nox)\equiv0_{\nuinput}$ and $\ell_f(\nox, \delta) = \frac{\lambda}{\rho_c}\delta^2$.
\end{proposition}

\begin{proof}
  Let $(\nox,\delta)\in\liXf$.
  From \cref{eq:terminal_nox_prop} and $k_f(\nox)\equiv 0_{\nuinput}$,
  $\noxnext=\Fapprox(\nox,k_f(\nox))=\Fapprox(0_{\nx}, 0_{\nuinput}) = 0_{\nx}$.
  Hence, \cref{eq:terminal_nox_prop} is preserved at the next time step.
  Moreover, since $\lift(0_{\nx})=0_{\nxlift}$, it holds that
  \begin{equation}
    \deltanext=\sqrt{1-\rho_c} \delta+ \sqrt{\alpha_2}( \cx\|\lift(\nox)\| + \cu\|k_f(\nox)\| )\le\sqrt{1-\rho_c}\delta, \label{eq:terminal_delta_prop_next}
  \end{equation}
  so \cref{eq:terminal_delta_prop} is preserved at the next time step.
  Hence, $(\noxnext,\deltanext)\in\liXf$, and \cref{eq:terminal_set} holds.
  Furthermore, \cref{eq:terminal_delta_prop} and \cref{eq:terminal_set} imply that \cref{eq:terminal_constraint} holds.
  Since \cref{eq:terminal_delta_prop_next} holds, we have
  \begin{align*}
    \ell_f(\nox, \delta) - \ell_f(\noxnext,\deltanext)\geq & \frac{\lambda}{\rho_c}\delta^2 - \frac{\lambda}{\rho_c}(\sqrt{1-\rho_c}\delta)^2 \\
    =                                                      & \lambda \delta^2 = \ell(\nox, k_f(\nox), \delta),
  \end{align*}
  so \cref{eq:terminal_cost} holds.
  Finally, for all $\hat \delta\in[0,\delta]$,
  \[ h_j(\nox,k_f(\nox))+c_j\hat \delta\le h_j(\nox,k_f(\nox))+c_j\delta\le0, \]
  so $(\nox, \hat \delta)\in\liXf$ holds and $\ell_f(\nox, \hat \delta) \leq \ell_f(\nox, \delta)$ holds, and \cref{eq:terminal_cost_monotone} holds.
\end{proof}
\begin{remark}
  The terminal set in \cref{prop:terminalset} is nonempty. Indeed,
  since $h_j(0_{\nx},0_{\nuinput}) < 0\, \forall j \in \integerset{1, n_h}$, there exists some $\bar\delta>0$ such that
  $h_j(0_{\nx},0_{\nuinput})+c_j\delta\le0$ holds for all $\delta\in[0,\bar\delta]$ and all $j\in\integerset{1,n_h}$ since $c_j\ge0$ for all $j\in\integerset{1,n_h}$.
  Hence, at least one $\delta\in\setR_{\ge0}$ satisfying \cref{eq:terminal_delta_prop} exists, and therefore $\liXf\ne\emptyset$.
\end{remark}


\subsection{Theoretical Analysis for Online Control}

Finally, we prove recursive feasibility and exponential stability for this TMPC framework.

\subsubsection*{Recursive Feasibility and Nominal Trajectory Convergence}
\begin{theorem}\label{thm:tmpc}
  Suppose \cref{assump:drccm} and \cref{ass:terminalset} hold and the proportional error bound~\cref{eq:error_bound_proportional} is valid. Let $\{\nox_{i|k}^*,\nou_{i|k}^*,\delta_{i|k}^*\mid i\in\integerset{0,\N}\}$ be the solution returned at time $k$ by~\cref{algorithm:mpc_online}.
  If~\cref{prob:tmpc} is feasible at time $k=0$, then~\cref{prob:tmpc} is feasible for all $k\ge0$. Moreover,
  $\lim_{k\to\infty}\nox^*_{0|k}=0_{\nx}$,
  $\lim_{k\to\infty}\nou^*_{0|k}=0_{\nuinput}$ and
  $\lim_{k\to\infty}\delta^*_{0|k}=0$ hold.
\end{theorem}

\begin{proof}
  The proof proceeds by induction on time $k$.
  Assume~\cref{prob:tmpc} is feasible at time $k$.
  For simplicity, we define $\nou^*_{\N|k}=k_f(\nox^*_{\N|k}),\nox^*_{\N+1|k}=\Fapprox(\nox^*_{\N|k},\nou^*_{\N|k})$, and $\delta^*_{\N+1|k}=\sqrt{1-\rho_c}\,\delta^*_{\N|k}+\sqrt{\alpha_2}( \cx\|\lift(\nox^*_{\N|k})\| + \cu\|\nou^*_{\N|k}\| )$.
  Construct a candidate at time $k+1$ by
  \begin{align*}
     & \nox_{0|k+1}=\nox^*_{1|k},\ \delta_{0|k+1}=V(\x(k+1),\nox_{0|k+1}), \\
     & \nou_{i|k+1}=\nou^*_{i+1|k},\ i\in\integerset{0,\N-1}
  \end{align*}
  and propagate the remaining terms by \cref{eq:nox,eq:delta}.
  Constraints \cref{eq:nox,eq:delta,eq:tubeinit} at time $k+1$ are immediate, and $\nox_{i|k+1} = \nox^*_{i+1|k}$ for all $i\in\integerset{0,\N}$ holds by construction.
  Next, we show that
  \begin{equation}\label{eq:delta_ineq}
    \delta_{i|k+1}\le\delta_{i+1|k}^*,\quad i\in\integerset{0,\N},
  \end{equation}
  by induction on $i$.
  From \cref{eq:delta,eq:tubeinit} at time $k$ and \cref{prop:tube_dynamics},
  $\delta_{0|k+1}=V(\x(k+1),\nox_{0|k+1})=V(\x(k+1),\nox^*_{1|k})\le\delta_{1|k}^*$, so $\delta_{0|k+1}\le\delta_{1|k}^*$ holds.
  Assume $\delta_{i|k+1}\le\delta_{i+1|k}^*$ for some $i\in\integerset{0,\N-1}$.
  From \cref{eq:delta} at time $k+1$,
  \begin{align*}
         & \delta_{i+1|k+1}                                                                                          \\
    =    & \sqrt{1-\rho_c}\,\delta_{i|k+1}+\sqrt{\alpha_2}( \cx\|\lift(\nox_{i|k+1})\| + \cu\|\nou_{i|k+1}\| )       \\
    =    & \sqrt{1-\rho_c}\,\delta_{i|k+1}+\sqrt{\alpha_2}( \cx\|\lift(\nox^*_{i+1|k})\| + \cu\|\nou^*_{i+1|k}\| )   \\
    \leq & \sqrt{1-\rho_c}\,\delta_{i+1|k}^*+\sqrt{\alpha_2}( \cx\|\lift(\nox^*_{i+1|k})\| + \cu\|\nou^*_{i+1|k}\| ) \\
    =    & \delta_{i+2|k}^*,
  \end{align*}
  so $\delta_{i+1|k+1}\le\delta_{i+2|k}^*$ holds. Thus, $\delta_{i|k+1}\le\delta_{i+1|k}^*$ for all $i\in\integerset{0,\N}$.
  Hence,
  \begin{align*}
        & h_j(\nox_{i|k+1},\nou_{i|k+1})+c_j\delta_{i|k+1}             \\
    \le & h_j(\nox^*_{i+1|k},\nou^*_{i+1|k})+c_j\delta^*_{i+1|k} \le 0
  \end{align*}
  holds for all $i \in \integerset{0,\N-2}$ and for all $j \in \integerset{1,n_h}$ by \cref{eq:constraint} at time $k$.
  For $i=\N-1$,~\eqref{eq:terminal} at time $k$ and \cref{ass:terminalset} imply
  \begin{align*}
        & h_j(\nox_{\N-1|k+1},\nou_{\N-1|k+1})+c_j\delta_{\N-1|k+1}  \\
    \le & h_j(\nox^*_{\N|k},\nou^*_{\N|k})+c_j\delta^*_{\N|k}  \leq0
  \end{align*}
  for all $j \in \integerset{1,n_h}$, so~\eqref{eq:constraint} holds at time $k+1$.
  And \cref{eq:terminal} at time $k+1$ follows from~\eqref{eq:terminal} at time $k$ and~\cref{ass:terminalset}.
  Thus, feasibility is recursive.

  Next, we prove convergence of nominal trajectories. Let $\totalcost_k$ be the value of the objective in~\cref{prob:tmpc} for $\{\nox_{i|k}^*,\nou_{i|k}^*,\delta_{i|k}^*\mid i\in\integerset{0,\N}\}$.
  It holds that
  \begin{equation}\label{eq:V_decrease}
    \begin{aligned}
           & \totalcost_{k+1}                                                                                                      \\
      \leq & \sum_{i=0}^{\N-1}\ell(\nox_{i|k+1},\nou_{i|k+1}, \delta_{i|k+1}) +\ell_f(\nox_{\N|k+1}, \delta_{\N|k+1})              \\
      \leq & \sum_{i=0}^{\N-1}\ell(\nox^*_{i+1|k},\nou^*_{i+1|k}, \delta^*_{i+1|k}) +\ell_f(\nox^*_{\N+1|k}, \delta^*_{\N+1|k})    \\
      =    & \totalcost_{k} - \ell(\nox^*_{0|k},\nou^*_{0|k}, \delta^*_{0|k}) + \ell(\nox^*_{\N|k},\nou^*_{\N|k}, \delta^*_{\N|k}) \\
           & \phantom{=\totalcost_{k}}  + \ell_f(\nox^*_{\N+1|k},\delta^*_{\N+1|k}) - \ell_f(\nox^*_{\N|k},\delta^*_{\N|k})        \\
      \leq & \totalcost_{k} - \ell(\nox^*_{0|k},\nou^*_{0|k}, \delta^*_{0|k}),
    \end{aligned}
  \end{equation}
  where the last inequality uses~\cref{eq:terminal_cost}.
  Summing the inequality~\cref{eq:V_decrease} over $k$ shows $\sum_{k=0}^\infty\ell(\nox_{0|k}^*,\nou_{0|k}^*, \delta^*_{0|k}) <\infty$.
  Therefore,
  $\underset{k\to\infty}{\lim}\ell(\nox_{0|k}^*,\nou_{0|k}^*, \delta^*_{0|k})=0$ holds.
  Because $\delta^*_{0|k} \geq 0$ holds and $\ell$ is defined as in~\eqref{eq:cost} with $Q\in \mathbb S_+^{\nx}$, $R\in \mathbb S_+^{\nuinput}$, and $\lambda \in \setR_{>0}$, it follows that $\nox_{0|k}^* \to 0_{\nx}$, $\nou_{0|k}^* \to 0_{\nuinput}$ and $\delta^*_{0|k} \to 0$ as $k\to\infty$.
\end{proof}

The construction of~\cref{prob:tmpc} and the statement of \cref{thm:tmpc} lead to the following corollary.

\begin{corollary}\label{cor:state_convergence}
  Under the conditions of \cref{thm:tmpc},
  if \cref{prob:tmpc} is feasible at time $k=0$,
  then the resulting closed-loop trajectory $\x(k)$ satisfies
  \begin{equation}\label{eq:state_convergence}
    \lim_{k\to\infty}\|\x(k)\| = 0.
  \end{equation}
\end{corollary}
\begin{proof}
  Let $\{\nox_{i|k}^*,\nou_{i|k}^*,\delta_{i|k}^*\mid i\in\integerset{0,\N}\}$ be a solution obtained by solving~\eqref{eq:mpc} at time $k$.
  For each $k$, \cref{eq:drccm_conditions} give
  $
    V(\x(k),\nox_{0|k}^*)
    \ge  \sqrt{\alpha_1}\|\x(k)-\nox_{0|k}^*\|
  $.
  Moreover, it holds that $\nox_{0|k}^* \to 0_{\nx}, \delta_{0|k}^* \to 0$ as $k\to\infty$.
  Therefore, by the triangle inequality and \cref{eq:tubeinit} , we have
  \begin{align*}
    \|\x(k)\| & \le \|\x(k)-\nox_{0|k}^*\| + \|\nox_{0|k}^*\|                         \\
              & \le \frac{1}{\sqrt{\alpha_1}}V(\x(k),\nox_{0|k}^*) + \|\nox_{0|k}^*\| \\
              & = \frac{1}{\sqrt{\alpha_1}}\delta_{0|k}^* + \|\nox_{0|k}^*\| \to 0.
  \end{align*}
\end{proof}

\subsubsection*{Exponential Stability of the Closed-Loop System}
In particular, when we use the terminal ingredients constructed in \cref{prop:terminalset}, the resulting closed-loop system is exponentially stable.

\begin{theorem}\label{thm:es}
  Under the conditions of \cref{thm:tmpc}, and if the terminal ingredients in~\cref{prob:tmpc} are constructed as in \cref{prop:terminalset}, then the resulting closed-loop system is exponentially stable on $\mathcal{X}_N$, where $\mathcal{X}_N$ is the set of initial states for which~\cref{prob:tmpc} is feasible.
\end{theorem}

\begin{proof}
  We prove that the cost sequence $\totalcost_k$ generated by~\cref{algorithm:mpc_online} satisfies quadratic bounds and a geometric decay estimate.
  Define $q$ as the minimum eigenvalue of $Q$ in~\cref{eq:cost}.
  From~\cref{eq:tubeinit} and~\cref{eq:drccm_conditions}, for every time $k$, we have
  \[ \delta_{0|k}^* = V(\x(k),\nox_{0|k}^*) \ge \sqrt{\alpha_1}\|\x(k)-\nox_{0|k}^*\|, \]
  so, it holds that
  \begin{align*}
        & \ell(\nox_{0|k}^*,\nou_{0|k}^*, \delta^*_{0|k})                                                                                                              \\
    \ge & q \|\nox_{0|k}^*\|^2 + \lambda \alpha_1 \|\x(k)-\nox_{0|k}^*\|^2                                                                                             \\
    \ge & (q + \lambda \alpha_1) \|\frac{\lambda \alpha_1}{q + \lambda \alpha_1}\x(k) - \nox_{0|k}^*\|^2  + \frac{q\lambda \alpha_1}{q + \lambda \alpha_1} \|\x(k)\|^2 \\
    \ge & \frac{q\lambda \alpha_1}{q + \lambda \alpha_1} \|\x(k)\|^2.
  \end{align*}
  Therefore, by combining with~\cref{eq:V_decrease}, we have
  \begin{equation}\label{eq:cost_decrease}
    \totalcost_k  \ge \gamma_1 \|\x(k)\|^2, \quad\totalcost_{k+1}  \le \totalcost_k - \gamma_1 \|\x(k)\|^2,
  \end{equation}
  where $\gamma_1 = \frac{q \lambda \alpha_1}{q + \lambda \alpha_1}$.
  Finally, we derive a quadratic upper bound for $\totalcost_k$.
  Since $h_j(0_{\nx},0_{\nuinput}) < 0$ for all $j \in \integerset{1, n_h}$, there exists a constant $r > 0$ such that if $\|\x(k)\| < r$, then $\{\nox_{i|k}^\text{loc},\nou_{i|k}^\text{loc},\delta_{i|k}^\text{loc}\mid i\in\integerset{0,\N}\}$ in~\cref{eq:local_solution} is a feasible candidate for~\cref{prob:tmpc} and $(s\x(k), 0_{\nuinput}) \in \Zsafe$ for all $s \in [0,1]$.
  The cost of this candidate is
  \begin{equation*}
    \begin{aligned}
        & \left(\lambda \left(\sum_{i=0}^{\N-1} (1-\rho_c)^{i}\right) + \frac{\lambda}{\rho_c} (1-\rho_c)^{\N}\right) V(\x(k),0_{\nx})^2 \\
      = & \frac{\lambda}{\rho_c} V(\x(k),0_{\nx})^2.
    \end{aligned}
  \end{equation*}
  Since~\cref{algorithm:mpc_online} selects a candidate whose cost is no larger than that of the feasible local solution in~\cref{eq:local_solution}, and by~\cref{eq:drccm_conditions}, it holds that
  \begin{equation}\label{eq:cost_upper_bound}
    \begin{aligned}
      \totalcost_k & \le \frac{\lambda}{\rho_c} V(\x(k),0_{\nx})^2                                                                      \\
                   & \le \frac{\lambda}{\rho_c} ( \int_0^1 \|\x(k)\|_{M(s\x(k))} ds )^2 \le \frac{\lambda \alpha_2}{\rho_c} \|\x(k)\|^2
    \end{aligned}
  \end{equation}
  for all $\x(k)$ which satisfy $\|\x(k)\| < r$.
  Next, by the compactness of $\Zsafe$, the continuity of $\ell, \Fapprox, \lift, V$, and the formulation of~\cref{prob:tmpc}, there exists a constant $\bar \totalcost$ such that $\totalcost_k \le \bar \totalcost$ holds for all $\x(k) \in \mathcal{X}_N$.
  Therefore, we have
  \begin{equation}\label{eq:cost_upper_bound_large}
    \totalcost_k \le \frac{\bar \totalcost}{r^2} \|\x(k)\|^2
  \end{equation}
  for all $\x(k)$ which satisfy $\|\x(k)\| \ge r$ and $\x(k) \in \mathcal{X}_N$.
  Combining~\cref{eq:cost_upper_bound,eq:cost_upper_bound_large}, for all $\x(k) \in \mathcal{X}_N$,
  we have
  \begin{equation}\label{eq:cost_upper_bound_final}
    \totalcost_k \le \gamma_2 \|\x(k)\|^2,
  \end{equation}
  where we choose $\gamma_2 = \max\{2\gamma_1, \frac{\lambda \alpha_2}{\rho_c}, \frac{\bar \totalcost}{r^2}\}$, so that $0<\gamma_1/\gamma_2<1$.
  Combining~\cref{eq:cost_upper_bound_final} with~\cref{eq:cost_decrease} yields
  \begin{equation}\label{eq:cost_decrease_final}
    \totalcost_{k+1}
    \le \totalcost_k-\gamma_1\|\x(k)\|^2
    \le \left(1-\frac{\gamma_1}{\gamma_2}\right)\totalcost_k.
  \end{equation}
  for all $\x(k) \in \mathcal{X}_N$. By~\cref{thm:tmpc}, we have $\x(k) \in \mathcal{X}_N$ for all $k\ge0$ when $\x(0) \in \mathcal{X}_N$ because of recursive feasibility.
  So, iterating~\cref{eq:cost_decrease_final} and using~\cref{eq:cost_decrease} and~\cref{eq:cost_upper_bound_final} gives
  \begin{equation*}
    \gamma_1\|\x(k)\|^2 \le \totalcost_k \le \left(1-\frac{\gamma_1}{\gamma_2}\right)^k\totalcost_0 \le \gamma_2\left(1-\frac{\gamma_1}{\gamma_2}\right)^k\|\x(0)\|^2
  \end{equation*}
  for all $k\ge0$ and $\x(0) \in \mathcal{X}_N$.
  This implies
  \begin{equation*}
    \|\x(k)\| \le \sqrt{\frac{\gamma_2}{\gamma_1}}\|\x(0)\|\left(\sqrt{1-\frac{\gamma_1}{\gamma_2}}\right)^k
  \end{equation*}
  for all $k\ge0$ and $\x(0) \in \mathcal{X}_N$, which shows that the closed-loop system is exponentially stable on $\mathcal{X}_N$.
\end{proof}

Combining these deterministic convergence results with the high-probability SafEDMD error certificate yields the following finite-data guarantee.

\begin{corollary}
  Let the data set $\dataset$ be generated according to the sampling scheme in~\cref{sec:safedmd} and let $\beta \in (0,1)$ be a probabilistic tolerance. Suppose $\K \geq \K_0$, \cref{assump:drccm} and \cref{ass:terminalset} hold, and \cref{prob:tmpc} is feasible at time $k=0$. Then, with probability at least $1-\beta$ with respect to $\dataset$, \cref{prob:tmpc} is feasible for all $k\ge0$, the original constraints satisfy
  \[
    (\x(k), \uinput(k)) \in \Zsafe, \quad \forall k\ge0,
  \]
  and
  \[ \lim_{k\to\infty}\|\x(k)\| = 0. \]
  In particular, if the terminal ingredients in~\cref{prob:tmpc} are constructed as in \cref{prop:terminalset}, then the resulting closed-loop system is exponentially stable on $\mathcal{X}_N$, where $\mathcal{X}_N$ is the set of initial states for which~\cref{prob:tmpc} is feasible.
\end{corollary}
Therefore, the proposed TMPC framework drives the sampled true state of the unknown system to the target state while satisfying the original nonlinear constraints at every sampling instant even under Koopman approximation errors without requiring a globally optimal solution to the MPC problem.

\section{Numerical Experiments}\label{section:numerical_experiment}

\subsection{Inverted Pendulum}\label{sec:inverted_pendulum}

\subsubsection*{System and Parameters}
We first consider an inverted pendulum:
\begin{equation}
  \begin{bmatrix}
    \dot{x}_1 \\
    \dot{x}_2
  \end{bmatrix}
  =
  \begin{bmatrix}
    x_2 \\
    \dfrac{g}{\ell} \sin x_1 - \dfrac{b}{m \ell^2} x_2 + \dfrac{1}{m \ell^2} u
  \end{bmatrix}
\end{equation}
with $g=\SI{9.81}{\meter\cdot\second^{-2}}$, $\ell=\SI{1}{\meter}$, $m=\SI{1}{\kilogram}$, and $b=\SI{0.01}{\newton\cdot\meter\cdot\second\cdot\radian^{-1}}$.
Here, $x=[x_1\ x_2]^\top\in\mathbb{R}^2$, where $x_1$ is the angle~($\si{\radian}$) and
$x_2$ is the angular velocity~($\si{\radian\cdot\second^{-1}}$), and $u\in\mathbb{R}$ is the control torque~($\si{\newton\cdot\meter}$).
The constraints are
$-[\SI{1}{\radian}\ \SI{2}{\radian\cdot\second^{-1}}]^\top\le x\le[\SI{1}{\radian}\ \SI{2}{\radian\cdot\second^{-1}}]^\top$ and $\SI{-20}{\newton\cdot\meter} \le u \le \SI{20}{\newton\cdot\meter}$.
The controller is required to steer the state from the initial condition
$\x(0)=[0.2\ 1]^\top$ to the target $\xref=[0\ 0]^\top$.

We discretize the continuous system with sampling time $T_s=\SI{0.01}{\second}$.
We use the observables $\lift(\x) = [x_1\ x_2\ \sin x_1]^\top$ and uniformly sample $d=10000$ points from the state-constraint set to learn a bilinear Koopman realization.
For~\eqref{eq:error_bound_proportional}, we set $\cx = \cu = 3 \times 10^{-4}$ following~\cite{Strasser+26}.
For \cref{algorithm:mpc_online}, we use
$Q=\mathrm{diag}\{1,1\}$, $R=0.001$, $\lambda= 100$, horizon $N=20$, and the computation-time limit $\tau=T_s$. Simulation time is $T=\SI{4}{\second}$, i.e., $400$ time steps.
Some of the computed parameters are listed in~\cref{appendix:parameter_setting_proposed}.

All computations were performed in MATLAB R2025b on macOS Tahoe 26.3
on an Apple M4 Max MacBook Pro with a 16-core CPU and \SI{128}{\giga\byte} memory.
The MPC problem was modeled in CasADi~\cite{Andersson+18} and solved by IPOPT~\cite{WB06}.

\subsubsection*{Simulation Result}
\Cref{fig:rmpc_trajectory} shows simulation results.
\begin{figure}[tb]
  \centering
  \includegraphics[width=\linewidth]{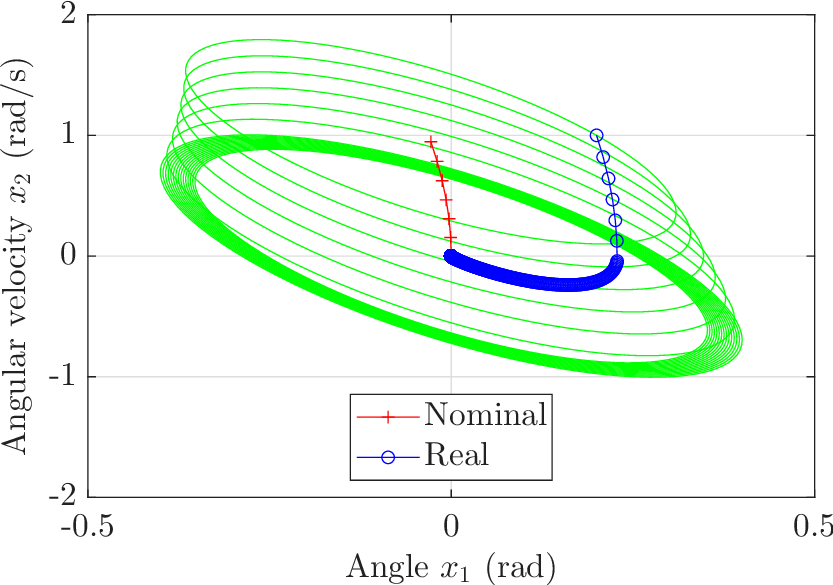}
  \caption{Closed-loop trajectory, nominal trajectory, and tubes under the proposed method for an inverted pendulum system.}
  \label{fig:rmpc_trajectory}
\end{figure}
Blue circles denote the true states, red plus signs denote nominal states from the MPC optimization at the first step, and green ellipses
indicate tubes around the nominal states.
The true state converges to the target while the tubes remain inside the state constraints, illustrating the constraint-satisfaction property of the proposed method.

\subsubsection*{Comparison with Related Works}
Furthermore, we compare the proposed method with the baseline methods~\cite{Schimperna+25, Schimperna+26, JL26}.
The plant, state and input constraints, initial condition, and reference are the same as above.
For the baseline methods~\cite{Schimperna+25, Schimperna+26, JL26}, we replace the original identification procedures with the SafEDMD procedure used above, employing the same observables $\hat{\Phi}(\x)$ and $d = 10,000$ training samples for each of the constant inputs $\uinput=0$ and $\uinput=1$. This reduces differences arising from the choice of identification method, rather than reproducing the complete learning-and-control pipelines of the original studies.
For the baselines requiring a constant approximation-error bound, we use the design bound
\begin{equation}\label{eq:constant_error_bound}
  \max_{(\x,\uinput)\in\Zsafe}(c_x\|\lift(\x)\|+c_u\|\uinput\|)=6.777 \times 10^{-3},
\end{equation}
where $\cx = \cu = 3 \times 10^{-4}$, as in the proposed method.
For RCCM-based TMPC in a lifted space~\cite{JL26}, as the authors mentioned in their paper, the computation of the terminal ingredients is practically difficult, so we set the terminal set to be the same as the lifted state constraints as in their numerical experiments.
Other parameters are listed in~\cref{appendix:parameter_setting}.

\Cref{table:inverted_pendulum_comparison} summarizes the total cost and optimization time of the proposed method and the baseline methods
and \cref{fig:inverted_pendulum_comparison} shows plots of the states and inputs over time.
Note that the optimization time is measured only for solver calls from the second step onward, and does not include solver calls in the first step, geodesic computations, or tube feedback law evaluations.
The panels show the proposed method~(\cref{fig:pendulum_proposed}), the methods in~\cite{Schimperna+25} and~\cite{Schimperna+26}~(\cref{fig:pendulum_schi25,fig:pendulum_schi26}, respectively), and RCCM-based TMPC in a lifted space~\cite{JL26}~(\cref{fig:pendulum_jl26}).
\begin{table*}[tb]
  \centering
  \caption{Total cost and optimization time of the proposed method and the baseline methods~\cite{Schimperna+25, Schimperna+26, JL26} for an inverted pendulum system. The proposed method's cost is obtained with the time-limit fallback enabled: the shifted feasible candidate is used whenever optimization exceeds $\tau=T_s=\SI{0.01}{\second}$. Its reported optimization times are hypothetical runtimes obtained by allowing optimization to run to completion without enforcing this time limit, rather than the runtimes under the time-limited control policy.}
  \label{table:inverted_pendulum_comparison}
  \begin{tabular}{c|ccc} 
    \hline 
    \multirow{2}{*}{Method}     & \multirow{2}{*}{Total Cost} & \multicolumn{2}{c}{Optimization Time ($\si{\ms}$)}                  \\
    \cline{3-4}
                                &                             & Max                                                & Average        \\
    \hline 
    Proposed                    & \textbf{11.367}             & 16.440                                             & 13.423         \\
    \mbox{\cite{Schimperna+25}} & 18.459                      & \textbf{8.829}                                     & 2.734          \\
    \mbox{\cite{Schimperna+26}} & 11.588                      & 12.394                                             & 3.882          \\
    \mbox{\cite{JL26}}          & 17.656                      & 15.776                                             & \textbf{1.896} \\
    \hline 
  \end{tabular}
\end{table*}

\begin{figure*}[tb]
  \begin{minipage}{0.48\linewidth}
    \centering
    \subfloat[Proposed\label{fig:pendulum_proposed}]{%
      \includegraphics[width=\linewidth]{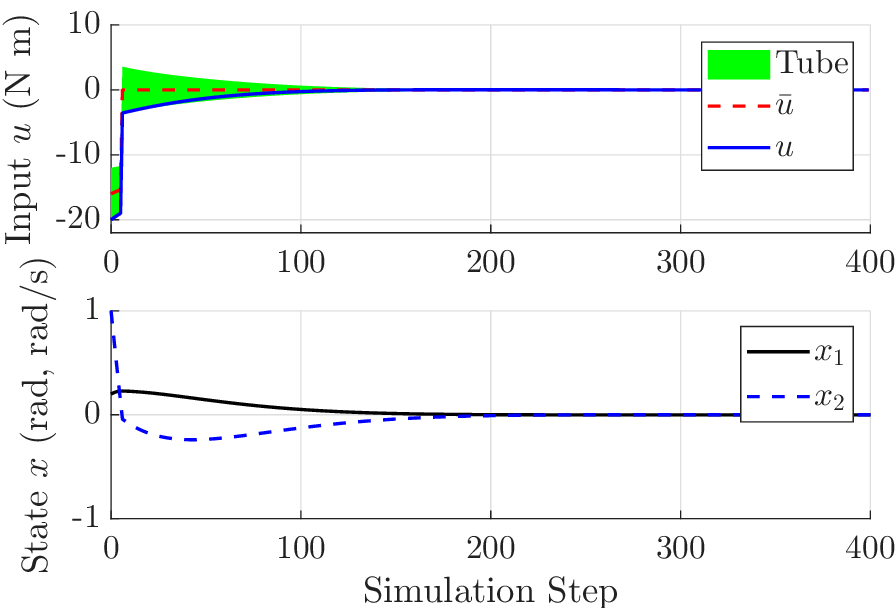}%
    }
  \end{minipage}
  \begin{minipage}{0.48\linewidth}
    \centering
    \subfloat[\cite{Schimperna+25}\label{fig:pendulum_schi25}]{%
      \includegraphics[width=\linewidth]{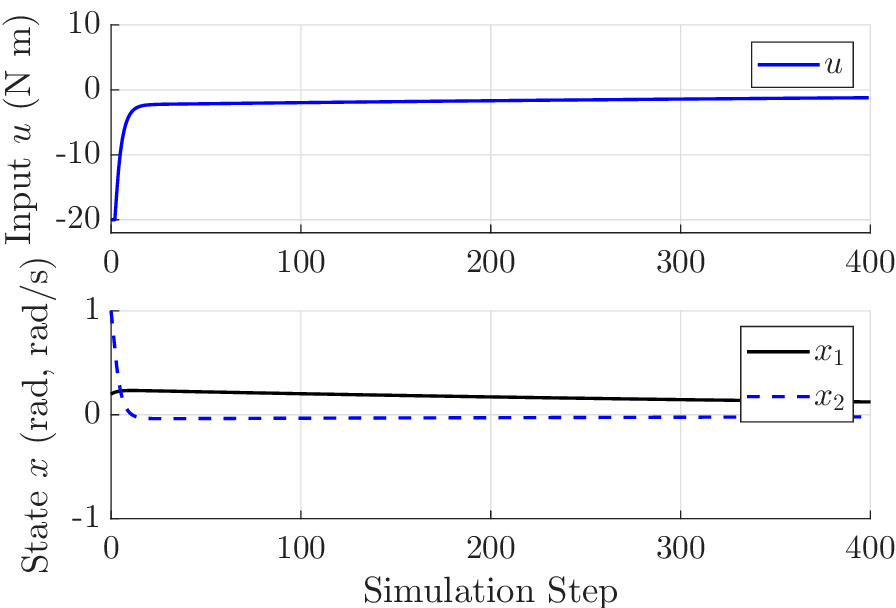}%
    }
  \end{minipage}

  \begin{minipage}{0.48\linewidth}
    \centering
    \subfloat[\cite{Schimperna+26}\label{fig:pendulum_schi26}]{%
      \includegraphics[width=\linewidth]{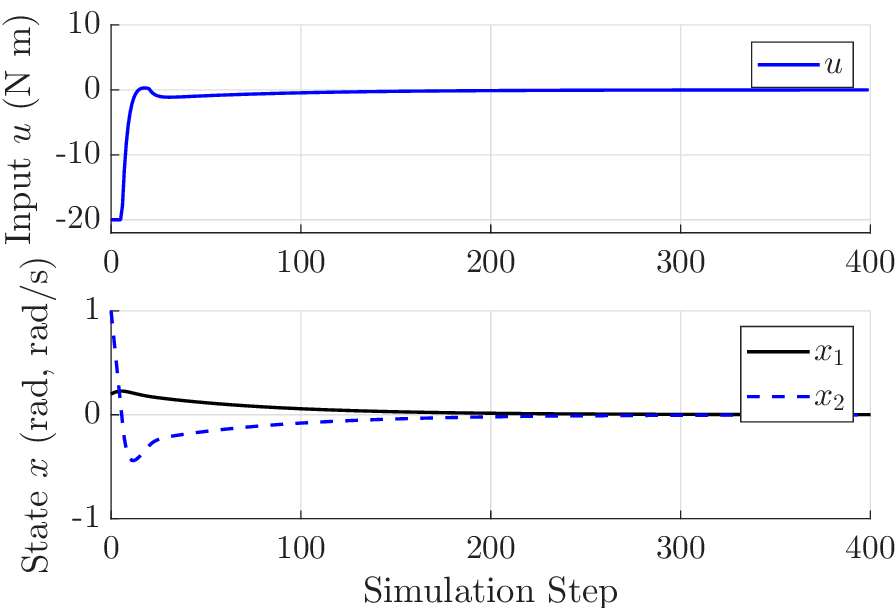}%
    }
  \end{minipage}
  \begin{minipage}{0.48\linewidth}
    \centering
    \subfloat[\cite{JL26}\label{fig:pendulum_jl26}]{%
      \includegraphics[width=\linewidth]{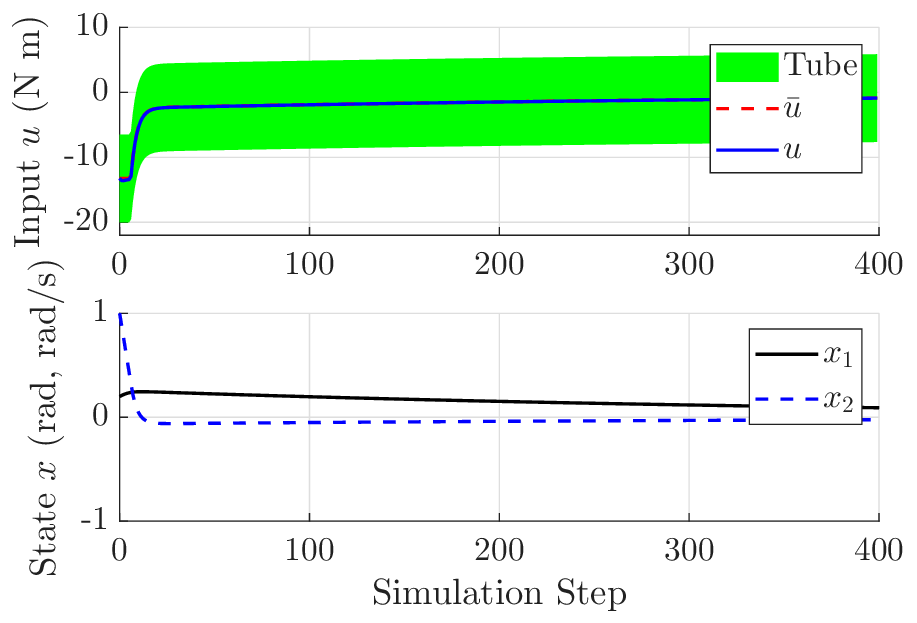}%
    }
  \end{minipage}
  \caption{Comparison of the proposed method with the baseline methods~\cite{Schimperna+25, Schimperna+26, JL26} implemented using a common SafEDMD model for an inverted pendulum system.}
  \label{fig:inverted_pendulum_comparison}
\end{figure*}
The total cost is defined as the sum of the stage cost for real states and inputs over the simulation time, i.e.,
\[\sum_{k=0}^{399} \left( \x(k)^\top Q \x(k) + \uinput(k)^\top R \uinput(k) \right).\]
The physical-state and input weights are $Q=I_2$ and $R=0.001$ respectively for all methods in the original state space~\cite{Schimperna+25, Schimperna+26}. For RCCM-based TMPC in a lifted space~\cite{JL26}, the weights are $Q=\mathrm{diag}\{1,1,0\}$ and $R=0.001$ for the lifted state and input respectively. All methods are evaluated using the same physical-state and input cost.
The maximum and average times in~\cref{table:inverted_pendulum_comparison} measure only MPC optimization, not the entire online control step.
For the proposed method, the maximum and average optimization times in the table represent hypothetical runtimes obtained by allowing the MPC optimization to run to completion without enforcing the time limit. They therefore do not represent the actual optimization runtimes under the time-limited control policy used to generate the closed-loop results: whenever optimization exceeds $\tau=T_s$, the controller uses the shifted feasible candidate instead of the new optimization result. The total cost of $11.367$ and the proposed method's trajectories in~\cref{fig:pendulum_proposed} are obtained with this fallback enabled.

\subsubsection*{Discussion}
In terms of cost, the proposed method outperforms the other methods, achieving the lowest cost.
For RCCM-based TMPC in a lifted space~\cite{JL26}, the tube tightening is overly conservative, resulting in smaller input magnitudes, as shown in~\cref{fig:pendulum_jl26}, and a higher cost.
In terms of optimization time, the proposed method is the slowest.
Although the formalized MPC optimization problem itself requires higher computational resources than baselines, the proposed method guarantees both closed-loop exponential stability and constraint satisfaction even when optimization is interrupted at any point, making it possible to safely control systems even under strict computational time constraints.
Due to limitations of the solver interface used in this implementation, we use the shifted feasible candidate whenever the time limit is exceeded. Further performance improvements may be possible by using a feasible solution found within the allotted time.

\subsection{Omni-directional Robot}\label{sec:omni}

\subsubsection*{System and Parameters}
We next consider an omni-directional robot with nonlinear obstacle-avoidance constraints:
\begin{equation}
  \begin{bmatrix}
    \dot{x} \\
    \dot{y} \\
    \dot{\theta}
  \end{bmatrix}
  =
  \begin{bmatrix}
    u_1 \cos \theta - u_2 \sin \theta \\
    u_1 \sin \theta + u_2 \cos \theta \\
    u_3
  \end{bmatrix},
\end{equation}
where $x \in \mathbb{R}$ is the $x$-position~($\si{\meter}$), $y \in \mathbb{R}$ is the $y$-position~($\si{\meter}$), $\theta \in \mathbb{R}$ is the heading angle~($\si{\radian}$), $u_1 \in \mathbb{R}$ is the forward velocity~($\si{\meter\cdot\second^{-1}}$), $u_2 \in \mathbb{R}$ is the lateral velocity~($\si{\meter\cdot\second^{-1}}$), and $u_3 \in \mathbb{R}$ is the angular velocity~($\si{\radian\cdot\second^{-1}}$).
The constraints are $-2\,\si{\meter} \le x, y \le 2\,\si{\meter}$, $-\dfrac{2\pi}{3}\,\si{\radian} \le \theta \le \dfrac{2\pi}{3}\,\si{\radian}$, $\SI{-2}{\meter\cdot\second^{-1}} \le u_1, u_2 \le \SI{2}{\meter\cdot\second^{-1}}$, and $-\dfrac{\pi}{3}\,\si{\radian\cdot\second^{-1}} \le u_3 \le \dfrac{\pi}{3}\,\si{\radian\cdot\second^{-1}}$.
We also impose the nonlinear obstacle-avoidance constraints $(x+\SI{0.5}{\meter})^2+(y+\SI{1.2}{\meter})^2\ge(\SI{0.3}{\meter})^2$, $(x+\SI{0.7}{\meter})^2+(y+\SI{0.4}{\meter})^2\ge(\SI{0.3}{\meter})^2$, and $(x-\SI{0.4}{\meter})^2+(y+\SI{1.1}{\meter})^2\ge(\SI{0.2}{\meter})^2$.
The controller is required to steer the state from the initial condition $\x(0)=[-1.8\ -1.8\ \dfrac{\pi}{2}]^\top$ to the target $\xref=[0\ 0\ 0]^\top$.

We discretize the continuous system with sampling time $T_s=\SI{0.1}{\second}$.
We use the observables $\lift(\x) = [x\ y\ \theta\ \sin \theta\ \cos \theta - 1]^\top$ and uniformly sample $d=10000$ points for each of the constant inputs $\uinput \in \{0, e_1, e_2, e_3\}$ from the admissible state region to learn a bilinear Koopman realization.
For~\eqref{eq:error_bound_proportional}, we set $\cx = 4.2 \times 10^{-4}$ and $\cu = 3.5 \times 10^{-3}$.
\begin{remark}
  Since the error bounds~\cref{eq:projection_error_bound} are valid for a general class of nonlinear systems, the resulting constants $\cx$, $\cu$ may be conservative, and even smaller constants may accurately capture the approximation error of the data-driven surrogate of the true system, as mentioned in the original SafEDMD paper~\cite[Rem.~5.1]{Strasser+26}.
  Hence, following the original procedure, we also use user-selected $\cx, \cu$.
  To evaluate the empirical coverage of the error bound with these constants, we generate $10000$ additional state--input pairs, independently and uniformly sampling the state from $[-2,2]^2\times[-2\pi/3,2\pi/3]$ and the input from $[-2,2]^2\times[-\pi/3,\pi/3]$. Following the validation implementation, we test $\|\xi_x\| < \cx\|\lift(\x)\|+\cu\|\uinput\|$, where $\xi_x$ is the one-step state prediction error. This inequality holds for $9140$ out of $10000$ validation samples, giving an empirical coverage of $91.4\%$.
\end{remark}
For \cref{algorithm:mpc_online}, we use
$Q=\mathrm{diag}\{1,1,1\}$, $R=\mathrm{diag}\{0.001,0.001,0.001\}$, $\lambda= 100$, horizon $N=20$, and the computation-time limit $\tau=T_s$.
The simulation runs for $200$ time steps, corresponding to $T=\SI{20}{\second}$ at $T_s=\SI{0.1}{\second}$. Some of the computed parameters are listed in~\cref{appendix:parameter_setting_proposed_omni}.

As in~\cref{sec:inverted_pendulum}, the proposed controller uses the shifted feasible candidate instead of the new optimization result whenever optimization exceeds $\tau=T_s=\SI{0.1}{\second}$. The trajectories shown for the proposed method are generated with this fallback enabled. All other computational conditions are the same as in~\cref{sec:inverted_pendulum}.

\subsubsection*{Simulation Result}
\Cref{fig:omni_tra} shows simulation results.
Blue circles denote the true positions $(x,y)$, black arrows denote the heading angles $\theta$ with lengths proportional to the magnitude of the velocity $\sqrt{u_1^2+u_2^2}$, red plus signs denote nominal states from the MPC optimization at the first step, green ellipses
indicate tubes around nominal states, and gray circles denote the obstacles.
The true state converges to the target, and no state or input constraint violations are observed along this simulated trajectory. The plotted tubes also avoid the obstacles.

\subsubsection*{Comparison with Related Works}
\Cref{fig:omni_cons_tra} shows the results obtained using a constant-error-bound variant based on our previous work~\cite{HS26}.
We use $\max_{(\x,\uinput)\in\Zsafe}{\cx\|\lift(\x)\|+\cu\|\uinput\|}$ as the uniform error bound.
\Cref{fig:omni_tra_schimperna26} shows the results obtained using the method in~\cite{Schimperna+26}.
Although the original paper assumes a convex state-constraint set, convexity is not needed for either the method or its theoretical guarantees. Under the remaining assumptions, removing the convexity assumption leaves these guarantees unchanged.
We therefore apply the method to the nonconvex state constraints in this experiment.
As in~\cref{sec:inverted_pendulum}, we replace the original identification procedure with the SafEDMD procedure used above, employing the same observables $\hat{\Phi}(\x)$ and $d = 10000$ training samples for each of the constant inputs $\uinput \in \{0, e_1, e_2, e_3\}$.

Unlike in~\cref{sec:inverted_pendulum}, we do not evaluate the other related works~\cite{Schimperna+25,JL26} in this experiment.
One of these methods~\cite{Schimperna+25} does not incorporate state constraints into the MPC optimization problem,
whereas the other~\cite{JL26} supports only convex constraints in the lifted space. Thus, neither method is directly applicable to the present setup with the chosen state constraints and observables.
\begin{figure*}[tb]
  \centering
  \begin{minipage}[t]{0.32\linewidth}
    \centering
    \subfloat[Proposed\label{fig:omni_tra}]{%
      \includegraphics[width=\linewidth]{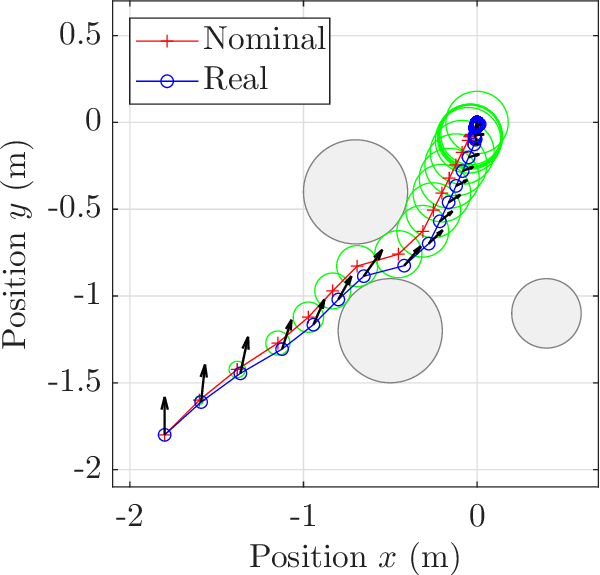}%
    }
  \end{minipage}\hfill
  \begin{minipage}[t]{0.32\linewidth}
    \centering
    \subfloat[Constant-error-bound variant~\cite{HS26}\label{fig:omni_cons_tra}]{%
      \includegraphics[width=\linewidth]{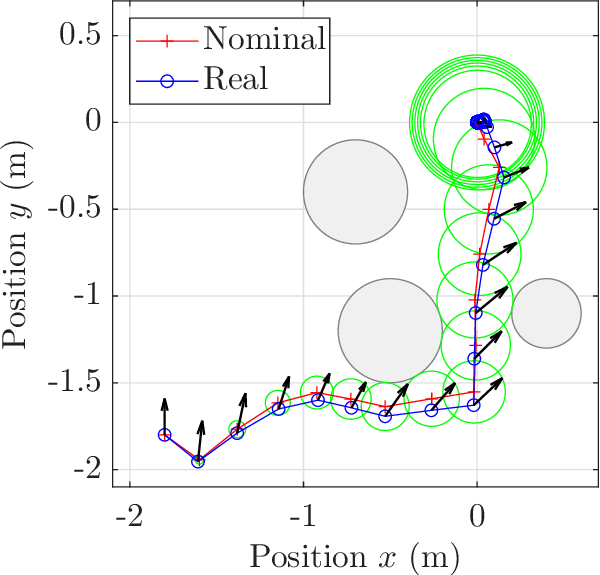}%
    }
  \end{minipage}\hfill
  \begin{minipage}[t]{0.32\linewidth}
    \centering
    \subfloat[\cite{Schimperna+26}\label{fig:omni_tra_schimperna26}]{%
      \includegraphics[width=\linewidth]{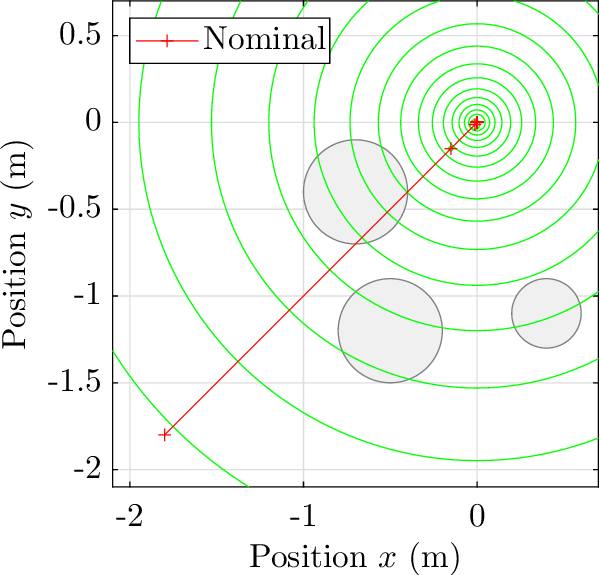}%
    }
  \end{minipage}
  \caption{Comparison of the proposed method with the constant-error-bound variant~\cite{HS26} and the method in~\cite{Schimperna+26} for an omni-directional robot with nonlinear obstacle-avoidance constraints.}
  \label{fig:omni_comparison}
\end{figure*}

\subsubsection*{Discussion}
These experiments demonstrate the advantages of the proposed method for nonlinear constraints like obstacle avoidance.
For the proposed method~(\cref{fig:omni_tra}), the planned trajectory passes through the narrow space between the large obstacles and reaches the target.
In contrast, the constant-error-bound variant based on our previous work~\cite{HS26}~(\cref{fig:omni_cons_tra}) does not use this narrow passage because the uniform error bound is overly conservative. Instead, the planned trajectory takes a less efficient route around the large obstacles.
For the method in~\cite{Schimperna+26}, as acknowledged by the authors, tube propagation is based on a constant error bound and Lipschitz constant without the contraction property resulting in the conservativeness of the constraint tightening with increasing prediction horizons.
Actually, the planned tube collides with the obstacles as shown in~\cref{fig:omni_tra_schimperna26}, and MPC optimization fails to find a feasible solution.
These results show that both the contraction property and the proportional error bound applied to the tube's time evolution contribute to its flexibility, enabling more efficient control in such strict nonlinear obstacle-avoidance constraints scenarios.

\section{Conclusion}\label{section:conclusion}

This paper has proposed a robust MPC framework for unknown nonlinear systems by combining data-driven bilinear Koopman realizations with contraction-based TMPC\@.
Using the proportional approximation-error bound obtained from SafEDMD, we constructed an error-aware predictor in the original state space through state-coordinate reprojection, avoiding the Koopman dictionary invariance assumption.
For this predictor, we developed a discrete-time RCCM-based homothetic tube construction and a tube-based MPC formulation with terminal ingredients and a tube-radius penalty.
We proved robust satisfaction of the original nonlinear constraints by the sampled closed-loop state-input pairs, recursive feasibility, and exponential stability of the sampled true state with respect to the target equilibrium with high probability over the training data without requiring a globally optimal solution to the MPC problem.
Numerical experiments, including a nonlinear obstacle-avoidance example, demonstrated robust stabilization, smaller closed-loop cost, and flexibility of the proposed method compared to existing Koopman-based MPC methods.
Future work will address the online computational cost caused by evaluating the lifting map inside the MPC optimization.
This issue may become significant when the lifting map is represented by a computationally expensive parametrization, such as a neural network~\cite{LKB18, YKH19, Haseli+26}, because the optimization may require evaluations of the lifting map, its derivatives, and possibly Hessians.
This direction would make the kernel-based Koopman realization~\cite{Strasser+25}, whose dimension scales with the number of samples, practical in our framework while ensuring true robustness through deterministic error bounds.

\appendices
\section{Parameter settings for the baseline methods in \cref{sec:inverted_pendulum}}\label[appendix]{appendix:parameter_setting}
\subsection{\cite{Schimperna+26}}
\begin{itemize}
  \item Uniform error bound~\cite[Ass.~1]{Schimperna+26}: $\eta^\epsilon = 6.777 \times 10^{-3}$ computed from~\cref{eq:constant_error_bound}
  \item Lipschitz constant~\cite[Ass.~1]{Schimperna+26}:  $\bar L =1.161$
  \item Terminal ingredients~\cite[p.988]{Schimperna+26}: $\beta = 1.5$, $c = 2.751$
\end{itemize}
\subsection{\cite{JL26}}
\begin{itemize}
  \item Disturbance bounds~\cite[(5), (6)]{JL26}: $\bar w =6.777 \times 10^{-3}$ computed from~\cref{eq:constant_error_bound} and $\bar \delta = 2 \bar w$
  \item Perturbation distribution~\cite[Algorithm 1]{JL26}: $\mathcal{U}[-\delta,  \delta]$ with $\delta = 0, 0.001, 0.0025, 0.005, 0.01$
  \item Hyperparameters for learning an RCCM~\cite[Problem 2]{JL26}: $\alpha = 0.1$, $\nu = 0.5$, $\beta = 5$, $\gamma_c = 1$, batch size 128, learning rate 0.001, and 10000 epochs
\end{itemize}
\section{Computed parameters for the proposed method in \cref{sec:inverted_pendulum}}\label[appendix]{appendix:parameter_setting_proposed}
We give the polynomial matrices $W=M^{-1}$ and $L=KW$, where $L$ corresponds to $Y$ in~\cref{rem:sos}.
The metric and differential feedback gain are recovered as
\begin{equation}\label{eq:appendix_metric_gain}
  M(\x)=W(\x)^{-1},\qquad K(\x)=L(\x)W(\x)^{-1}.
\end{equation}
For the inverted pendulum, define
\[
  p(\x)=[1\;x_1\;x_2\;x_1^2\;x_1x_2\;x_2^2]^\top.
\]
The entries are $W_{ij}(\x)=w_{ij}p(\x)$ and $L_{1j}(\x)=l_{1j}p(\x)$, with $W_{21}=W_{12}$ and the following row coefficient vectors
  {\small
\begin{align*}
  w_{11} & =[\begin{aligned}[t]
              &2.834\times10^{0}\;1.003\times10^{-3}\;-8.008\times10^{-4}\\
              &-3.709\times10^{-2}\;1.565\times10^{-2}\;3.311\times10^{-4}],
  \end{aligned}                                                                \\
  w_{12} & =[\begin{aligned}[t]
              &-4.928\times10^{0}\;-1.433\times10^{-3}\;1.545\times10^{-3}\\
              &6.844\times10^{-2}\;-5.971\times10^{-2}\;1.429\times10^{-4}],
  \end{aligned}                                                                \\
  w_{22} & =[\begin{aligned}[t]
              &1.782\times10^{1}\;-8.063\times10^{-4}\;-1.760\times10^{-3}\\
              &6.600\times10^{-3}\;2.359\times10^{-1}\;2.552\times10^{-3}],
  \end{aligned}                                                                \\
  l_{11} & =[\begin{aligned}[t]
              &-2.418\times10^{1}\;-4.108\times10^{-4}\;-3.753\times10^{-4}\\
              &2.815\times10^{-3}\;5.325\times10^{-4}\;-8.174\times10^{-4}],
  \end{aligned}                                                                \\
  l_{12} & =[\begin{aligned}[t]
              &-2.636\times10^{-3}\;1.960\times10^{-4}\;-2.252\times10^{-5}\\
              &-1.092\times10^{-3}\;-2.285\times10^{-4}\;-8.693\times10^{-4}],
  \end{aligned}
\end{align*}

  }
The contraction parameter and numerical metric bounds are
\[
  \rho_c=0.0225,\qquad \alpha_1\simeq0.0505635,\qquad
  \alpha_2\simeq0.736173.
\]
Here and in~\cref{appendix:parameter_setting_proposed_omni}, $\alpha_1$ is the smallest eigenvalue of the common lower-bound matrix, and $\alpha_2$ is the maximum eigenvalue of $M$ over the evaluation grid.
These metric bounds and the constraint-tightening constants below are numerical grid-based values, rather than certified bounds over the continuous domain.
For the constraint ordering
\[
  h(\x,\uinput)=[x_1-1\;x_2-2\;-x_1-1\;-x_2-2\;u-20\;-u-20]^\top,
\]
the tightening constants are
\[
  (c_1,\ldots,c_6)\simeq
  \begin{aligned}[t]
    ( & 3.404949,\;3.090077,\;3.404949,    \\
      & 3.090077,\;19.935587,\;19.935587).
  \end{aligned}
\]

\section{Computed parameters for the proposed method in \cref{sec:omni}}\label[appendix]{appendix:parameter_setting_proposed_omni}
We give the polynomial matrices $W$ and $L$, from which $M$ and $K$ are recovered using~\eqref{eq:appendix_metric_gain}.
The matrices $W$ and $L$ contain monomials in $(x,y,\theta)$ of total degree at most four.
Let $p_o(\x)\in\mathbb R^{35}$ contain all these monomials in the order shown in the first column of~\cref{table:omni_w_coefficients,table:omni_l_coefficients}.
The entries are $W_{ij}(\x)=w_{ij}p_o(\x)$ and $L_{ij}(\x)=l_{ij}p_o(\x)$.
The tables list each row coefficient vector vertically, rounded to four significant figures; small nonzero coefficients are retained in scientific notation.
The matrices satisfy $W_{ji}=W_{ij}$ and $L_{ji}=L_{ij}$, so only their upper-triangular entries are listed.
\begin{table*}[tb]
  \centering
  \caption{Omni-directional robot: entries of the row coefficient vectors $w_{ij}$, in basis order (four significant figures).}
  \label{table:omni_w_coefficients}
  \scriptsize
  \setlength{\tabcolsep}{3pt}
  \begin{tabular}{c|rrrrrr}
    \hline
    $p_{o,k}$         & $w_{11}[k]$            & $w_{12}[k]$            & $w_{13}[k]$            & $w_{22}[k]$            & $w_{23}[k]$            & $w_{33}[k]$            \\ \hline
    $1$               & $3.689\times10^{0}$    & $-1.873\times10^{-10}$ & $-5.644\times10^{-4}$  & $3.682\times10^{0}$    & $5.003\times10^{-4}$   & $3.686\times10^{0}$    \\
    $x$               & $-5.780\times10^{-12}$ & $3.870\times10^{-12}$  & $-6.786\times10^{-13}$ & $-1.436\times10^{-11}$ & $6.260\times10^{-13}$  & $-1.895\times10^{-11}$ \\
    $y$               & $6.889\times10^{-12}$  & $-6.337\times10^{-12}$ & $2.170\times10^{-11}$  & $3.339\times10^{-12}$  & $-2.373\times10^{-11}$ & $1.293\times10^{-11}$  \\
    $\theta$          & $1.455\times10^{-11}$  & $-1.281\times10^{-12}$ & $4.064\times10^{-13}$  & $1.122\times10^{-12}$  & $-4.044\times10^{-12}$ & $9.547\times10^{-12}$  \\
    $x^{2}$           & $6.339\times10^{-11}$  & $-3.303\times10^{-11}$ & $1.263\times10^{-11}$  & $1.181\times10^{-10}$  & $-2.089\times10^{-11}$ & $7.925\times10^{-11}$  \\
    $xy$              & $1.352\times10^{-11}$  & $-4.981\times10^{-12}$ & $2.168\times10^{-12}$  & $5.252\times10^{-12}$  & $-3.400\times10^{-12}$ & $1.841\times10^{-12}$  \\
    $y^{2}$           & $1.213\times10^{-11}$  & $-9.784\times10^{-12}$ & $7.051\times10^{-12}$  & $2.402\times10^{-11}$  & $-1.196\times10^{-11}$ & $5.688\times10^{-11}$  \\
    $x\theta$         & $2.950\times10^{-11}$  & $-1.834\times10^{-11}$ & $2.077\times10^{-11}$  & $2.821\times10^{-11}$  & $-3.068\times10^{-11}$ & $1.044\times10^{-11}$  \\
    $y\theta$         & $-2.467\times10^{-11}$ & $2.246\times10^{-11}$  & $-1.214\times10^{-11}$ & $-2.584\times10^{-11}$ & $5.466\times10^{-12}$  & $1.101\times10^{-11}$  \\
    $\theta^{2}$      & $-1.779\times10^{-11}$ & $3.623\times10^{-11}$  & $-4.787\times10^{-11}$ & $-1.558\times10^{-8}$  & $2.418\times10^{-8}$   & $3.287\times10^{-12}$  \\
    $x^{3}$           & $8.069\times10^{-12}$  & $-3.390\times10^{-12}$ & $3.611\times10^{-12}$  & $3.091\times10^{-12}$  & $-3.838\times10^{-12}$ & $1.888\times10^{-12}$  \\
    $x^{2}y$          & $-1.774\times10^{-11}$ & $1.315\times10^{-11}$  & $-6.052\times10^{-12}$ & $-3.456\times10^{-11}$ & $6.251\times10^{-12}$  & $-1.203\times10^{-11}$ \\
    $xy^{2}$          & $-5.986\times10^{-12}$ & $5.963\times10^{-12}$  & $-2.223\times10^{-12}$ & $-1.721\times10^{-11}$ & $1.798\times10^{-12}$  & $3.452\times10^{-14}$  \\
    $y^{3}$           & $-1.703\times10^{-11}$ & $1.705\times10^{-11}$  & $-9.718\times10^{-12}$ & $-8.500\times10^{-11}$ & $1.168\times10^{-11}$  & $-3.052\times10^{-11}$ \\
    $x^{2}\theta$     & $-2.608\times10^{-11}$ & $1.686\times10^{-11}$  & $-8.202\times10^{-12}$ & $-6.068\times10^{-11}$ & $1.085\times10^{-11}$  & $-1.143\times10^{-11}$ \\
    $xy\theta$        & $-5.430\times10^{-12}$ & $7.205\times10^{-12}$  & $-3.353\times10^{-12}$ & $-2.965\times10^{-11}$ & $3.440\times10^{-12}$  & $-3.607\times10^{-12}$ \\
    $y^{2}\theta$     & $1.461\times10^{-11}$  & $-4.190\times10^{-12}$ & $4.099\times10^{-14}$  & $8.436\times10^{-12}$  & $-4.006\times10^{-12}$ & $-3.564\times10^{-12}$ \\
    $x\theta^{2}$     & $3.510\times10^{-11}$  & $-1.108\times10^{-11}$ & $3.556\times10^{-12}$  & $1.985\times10^{-11}$  & $-1.207\times10^{-11}$ & $1.539\times10^{-11}$  \\
    $y\theta^{2}$     & $5.897\times10^{-11}$  & $-1.206\times10^{-10}$ & $3.720\times10^{-11}$  & $-1.076\times10^{-12}$ & $-3.529\times10^{-11}$ & $4.021\times10^{-11}$  \\
    $\theta^{3}$      & $9.369\times10^{-9}$   & $-2.283\times10^{-11}$ & $-2.496\times10^{-11}$ & $1.089\times10^{-9}$   & $-6.986\times10^{-13}$ & $2.696\times10^{-11}$  \\
    $x^{4}$           & $2.126\times10^{-11}$  & $-1.333\times10^{-11}$ & $7.700\times10^{-12}$  & $3.891\times10^{-11}$  & $-1.196\times10^{-11}$ & $2.688\times10^{-11}$  \\
    $x^{3}y$          & $-7.201\times10^{-12}$ & $5.816\times10^{-12}$  & $-4.088\times10^{-12}$ & $-1.692\times10^{-11}$ & $4.878\times10^{-12}$  & $-8.529\times10^{-12}$ \\
    $x^{2}y^{2}$      & $1.431\times10^{-11}$  & $-6.327\times10^{-12}$ & $5.125\times10^{-12}$  & $1.407\times10^{-11}$  & $-8.368\times10^{-12}$ & $1.646\times10^{-11}$  \\
    $xy^{3}$          & $-6.169\times10^{-13}$ & $2.568\times10^{-12}$  & $-1.007\times10^{-12}$ & $-9.832\times10^{-12}$ & $4.352\times10^{-13}$  & $-1.744\times10^{-12}$ \\
    $y^{4}$           & $5.946\times10^{-11}$  & $-2.670\times10^{-11}$ & $1.553\times10^{-11}$  & $1.249\times10^{-10}$  & $-3.088\times10^{-11}$ & $8.882\times10^{-11}$  \\
    $x^{3}\theta$     & $-9.909\times10^{-12}$ & $8.558\times10^{-12}$  & $-5.318\times10^{-12}$ & $-2.781\times10^{-11}$ & $6.681\times10^{-12}$  & $-1.571\times10^{-11}$ \\
    $x^{2}y\theta$    & $6.652\times10^{-12}$  & $-2.332\times10^{-12}$ & $7.645\times10^{-12}$  & $-2.017\times10^{-12}$ & $-9.510\times10^{-12}$ & $1.627\times10^{-11}$  \\
    $xy^{2}\theta$    & $-2.564\times10^{-12}$ & $5.279\times10^{-12}$  & $1.910\times10^{-12}$  & $-2.561\times10^{-11}$ & $-2.879\times10^{-12}$ & $1.957\times10^{-12}$  \\
    $y^{3}\theta$     & $1.171\times10^{-11}$  & $-7.302\times10^{-12}$ & $6.238\times10^{-12}$  & $2.344\times10^{-11}$  & $-9.229\times10^{-12}$ & $2.419\times10^{-11}$  \\
    $x^{2}\theta^{2}$ & $2.525\times10^{-11}$  & $-3.704\times10^{-12}$ & $2.215\times10^{-12}$  & $7.569\times10^{-12}$  & $-1.033\times10^{-11}$ & $3.661\times10^{-11}$  \\
    $xy\theta^{2}$    & $2.031\times10^{-11}$  & $-1.271\times10^{-12}$ & $2.871\times10^{-12}$  & $-1.435\times10^{-11}$ & $-9.836\times10^{-12}$ & $1.603\times10^{-11}$  \\
    $y^{2}\theta^{2}$ & $6.083\times10^{-12}$  & $5.196\times10^{-12}$  & $1.115\times10^{-11}$  & $-8.741\times10^{-11}$ & $-1.472\times10^{-11}$ & $3.276\times10^{-11}$  \\
    $x\theta^{3}$     & $1.344\times10^{-10}$  & $-3.200\times10^{-11}$ & $2.238\times10^{-11}$  & $1.851\times10^{-11}$  & $-4.569\times10^{-11}$ & $4.732\times10^{-11}$  \\
    $y\theta^{3}$     & $3.967\times10^{-11}$  & $5.527\times10^{-12}$  & $-5.604\times10^{-11}$ & $6.478\times10^{-11}$  & $1.273\times10^{-12}$  & $2.304\times10^{-11}$  \\
    $\theta^{4}$      & $1.781\times10^{-10}$  & $-3.993\times10^{-8}$  & $1.564\times10^{-9}$   & $-5.209\times10^{-8}$  & $3.336\times10^{-10}$  & $7.527\times10^{-8}$   \\
    \hline
  \end{tabular}
\end{table*}
\begin{table*}[tb]
  \centering
  \caption{Omni-directional robot: entries of the row coefficient vectors $l_{ij}$, in basis order (four significant figures).}
  \label{table:omni_l_coefficients}
  \scriptsize
  \setlength{\tabcolsep}{3pt}
  \begin{tabular}{c|rrrrrr}
    \hline
    $p_{o,k}$         & $l_{11}[k]$            & $l_{12}[k]$            & $l_{13}[k]$            & $l_{22}[k]$            & $l_{23}[k]$            & $l_{33}[k]$            \\ \hline
    $1$               & $-1.338\times10^{0}$   & $1.783\times10^{-3}$   & $-1.466\times10^{-11}$ & $-1.339\times10^{0}$   & $3.061\times10^{-11}$  & $-9.799\times10^{-1}$  \\
    $x$               & $-2.606\times10^{-12}$ & $-6.361\times10^{-13}$ & $-4.673\times10^{-13}$ & $7.556\times10^{-12}$  & $-1.339\times10^{-12}$ & $1.259\times10^{-12}$  \\
    $y$               & $1.992\times10^{-11}$  & $-4.876\times10^{-12}$ & $3.305\times10^{-12}$  & $2.607\times10^{-12}$  & $-1.737\times10^{-12}$ & $5.115\times10^{-12}$  \\
    $\theta$          & $1.232\times10^{-10}$  & $-3.824\times10^{-12}$ & $3.267\times10^{-12}$  & $-1.805\times10^{-11}$ & $-1.973\times10^{-13}$ & $3.886\times10^{-12}$  \\
    $x^{2}$           & $2.821\times10^{-13}$  & $5.319\times10^{-13}$  & $-4.063\times10^{-13}$ & $-3.792\times10^{-12}$ & $9.744\times10^{-13}$  & $-1.792\times10^{-12}$ \\
    $xy$              & $2.538\times10^{-12}$  & $-9.368\times10^{-13}$ & $4.753\times10^{-13}$  & $2.920\times10^{-13}$  & $-3.723\times10^{-13}$ & $9.107\times10^{-13}$  \\
    $y^{2}$           & $-5.022\times10^{-12}$ & $1.449\times10^{-12}$  & $-1.680\times10^{-12}$ & $-5.386\times10^{-14}$ & $1.144\times10^{-12}$  & $-3.698\times10^{-12}$ \\
    $x\theta$         & $3.919\times10^{-12}$  & $-2.627\times10^{-12}$ & $8.979\times10^{-13}$  & $5.648\times10^{-12}$  & $-1.839\times10^{-12}$ & $3.232\times10^{-12}$  \\
    $y\theta$         & $-5.346\times10^{-12}$ & $1.659\times10^{-12}$  & $-2.346\times10^{-12}$ & $-4.059\times10^{-12}$ & $8.679\times10^{-13}$  & $-2.861\times10^{-12}$ \\
    $\theta^{2}$      & $-4.057\times10^{-12}$ & $2.030\times10^{-12}$  & $-3.561\times10^{-12}$ & $-4.057\times10^{-12}$ & $5.273\times10^{-12}$  & $-1.807\times10^{-11}$ \\
    $x^{3}$           & $3.552\times10^{-13}$  & $-5.084\times10^{-13}$ & $6.635\times10^{-14}$  & $1.633\times10^{-12}$  & $-4.484\times10^{-13}$ & $6.320\times10^{-13}$  \\
    $x^{2}y$          & $9.030\times10^{-13}$  & $-3.537\times10^{-13}$ & $5.969\times10^{-14}$  & $-5.615\times10^{-14}$ & $-7.145\times10^{-14}$ & $1.033\times10^{-13}$  \\
    $xy^{2}$          & $-2.288\times10^{-13}$ & $7.778\times10^{-15}$  & $-3.011\times10^{-13}$ & $-2.424\times10^{-13}$ & $1.378\times10^{-13}$  & $-5.469\times10^{-13}$ \\
    $y^{3}$           & $4.112\times10^{-12}$  & $-1.486\times10^{-12}$ & $8.935\times10^{-13}$  & $1.370\times10^{-12}$  & $-4.578\times10^{-13}$ & $1.133\times10^{-12}$  \\
    $x^{2}\theta$     & $1.880\times10^{-12}$  & $-2.773\times10^{-13}$ & $-1.027\times10^{-13}$ & $-2.550\times10^{-12}$ & $4.942\times10^{-13}$  & $-8.985\times10^{-13}$ \\
    $xy\theta$        & $4.340\times10^{-13}$  & $-3.158\times10^{-13}$ & $-4.035\times10^{-13}$ & $-3.157\times10^{-13}$ & $1.266\times10^{-13}$  & $-7.163\times10^{-13}$ \\
    $y^{2}\theta$     & $6.615\times10^{-12}$  & $-2.134\times10^{-12}$ & $1.011\times10^{-12}$  & $-1.657\times10^{-12}$ & $-1.141\times10^{-12}$ & $3.119\times10^{-12}$  \\
    $x\theta^{2}$     & $2.092\times10^{-12}$  & $-1.816\times10^{-12}$ & $-3.492\times10^{-13}$ & $2.831\times10^{-12}$  & $-8.027\times10^{-13}$ & $4.237\times10^{-13}$  \\
    $y\theta^{2}$     & $1.338\times10^{-11}$  & $-4.790\times10^{-12}$ & $1.425\times10^{-12}$  & $5.614\times10^{-12}$  & $-3.414\times10^{-13}$ & $-7.398\times10^{-14}$ \\
    $\theta^{3}$      & $1.585\times10^{-3}$   & $-7.486\times10^{-12}$ & $8.173\times10^{-13}$  & $-2.107\times10^{-10}$ & $-2.137\times10^{-12}$ & $7.149\times10^{-12}$  \\
    $x^{4}$           & $4.246\times10^{-13}$  & $-1.651\times10^{-14}$ & $-7.393\times10^{-14}$ & $-8.755\times10^{-13}$ & $1.987\times10^{-13}$  & $-3.799\times10^{-13}$ \\
    $x^{3}y$          & $5.243\times10^{-13}$  & $-2.261\times10^{-13}$ & $-1.671\times10^{-15}$ & $-2.248\times10^{-14}$ & $-3.689\times10^{-14}$ & $7.147\times10^{-15}$  \\
    $x^{2}y^{2}$      & $3.976\times10^{-13}$  & $-1.875\times10^{-13}$ & $-6.733\times10^{-14}$ & $-8.843\times10^{-14}$ & $3.014\times10^{-15}$  & $-1.194\times10^{-13}$ \\
    $xy^{3}$          & $1.182\times10^{-12}$  & $-4.599\times10^{-13}$ & $1.274\times10^{-13}$  & $1.645\times10^{-14}$  & $-1.269\times10^{-13}$ & $2.426\times10^{-13}$  \\
    $y^{4}$           & $-9.275\times10^{-13}$ & $1.859\times10^{-13}$  & $-4.979\times10^{-13}$ & $4.737\times10^{-13}$  & $4.036\times10^{-13}$  & $-1.387\times10^{-12}$ \\
    $x^{3}\theta$     & $1.100\times10^{-12}$  & $-7.135\times10^{-13}$ & $5.633\times10^{-14}$  & $1.077\times10^{-12}$  & $-3.701\times10^{-13}$ & $4.659\times10^{-13}$  \\
    $x^{2}y\theta$    & $1.075\times10^{-12}$  & $-4.787\times10^{-13}$ & $-1.078\times10^{-13}$ & $-2.010\times10^{-13}$ & $-1.585\times10^{-14}$ & $-1.962\times10^{-13}$ \\
    $xy^{2}\theta$    & $2.191\times10^{-12}$  & $-8.784\times10^{-13}$ & $1.369\times10^{-13}$  & $-3.960\times10^{-14}$ & $-1.907\times10^{-13}$ & $2.624\times10^{-13}$  \\
    $y^{3}\theta$     & $-3.553\times10^{-13}$ & $6.396\times10^{-14}$  & $-7.438\times10^{-13}$ & $-2.549\times10^{-12}$ & $8.935\times10^{-15}$  & $-2.789\times10^{-13}$ \\
    $x^{2}\theta^{2}$ & $2.793\times10^{-12}$  & $-8.673\times10^{-13}$ & $-2.926\times10^{-13}$ & $-2.069\times10^{-12}$ & $3.465\times10^{-13}$  & $-1.019\times10^{-12}$ \\
    $xy\theta^{2}$    & $4.541\times10^{-12}$  & $-1.846\times10^{-12}$ & $1.193\times10^{-13}$  & $-1.982\times10^{-13}$ & $-3.277\times10^{-13}$ & $2.712\times10^{-13}$  \\
    $y^{2}\theta^{2}$ & $1.274\times10^{-12}$  & $-1.093\times10^{-12}$ & $-8.193\times10^{-13}$ & $3.669\times10^{-12}$  & $9.951\times10^{-13}$  & $-4.123\times10^{-12}$ \\
    $x\theta^{3}$     & $1.214\times10^{-11}$  & $-5.604\times10^{-12}$ & $1.798\times10^{-13}$  & $1.797\times10^{-12}$  & $-1.204\times10^{-12}$ & $1.102\times10^{-12}$  \\
    $y\theta^{3}$     & $6.673\times10^{-12}$  & $-2.418\times10^{-12}$ & $-1.668\times10^{-12}$ & $-2.347\times10^{-11}$ & $-1.377\times10^{-12}$ & $2.712\times10^{-12}$  \\
    $\theta^{4}$      & $1.780\times10^{-1}$   & $-9.359\times10^{-5}$  & $-1.684\times10^{-12}$ & $1.782\times10^{-1}$   & $5.743\times10^{-12}$  & $-9.536\times10^{-4}$  \\
    \hline
  \end{tabular}
\end{table*}

The contraction parameter and numerical metric bounds are
\[
  \rho_c=0.05,\qquad \alpha_1\simeq0.2710559,\qquad
  \alpha_2\simeq0.2715938.
\]

To specify the scaling and ordering of $c_j$, let
$\bar{\x}=[2\;2\;2\pi/3]^\top$ and
$\bar{\uinput}=[2\;2\;\pi/3]^\top$.
The first twelve constraints are ordered as
\[
  h_{1:12}(\x,\uinput)=
  [ (\x-\bar{\x})^\top\;(-\x-\bar{\x})^\top\;
    (\uinput-\bar{\uinput})^\top\;(-\uinput-\bar{\uinput})^\top]^\top.
\]
For obstacles with centers $o_i$ and radii $r_i$, in the order given in~\cref{sec:omni}, the implementation uses the equivalent distance-form constraints
\[
  h_{12+i}(\x,\uinput)=r_i-\|[x\;y]^\top-o_i\|,\qquad i=1,2,3.
\]
The corresponding tightening constants are
\begin{align*}
  c_1=\cdots=c_6         & \simeq1.920681,                         \\
  (c_7,c_8,c_9)          & \simeq(1.094520,\;1.089492,\;0.519961), \\
  (c_{10},c_{11},c_{12}) & =(c_7,c_8,c_9),                         \\
  (c_{13},c_{14},c_{15}) & \simeq(1.920533,\;1.920531,\;1.920190).
\end{align*}

\section*{References}
\bibliographystyle{IEEEtran}
\bibliography{reference}

%
\end{document}